\documentclass[a4paper,twoside,11pt,reqno]{amsart}
\newcommand{\Ueberschrift}{Simply transitive quaternionic lattices of rank $2$ over $\bF_q(t)$ and \\
a non-classical fake quadric}
\newcommand{\Kurztitel}{Quaternionic lattices and non-classical fake quadric}
\usepackage{amsmath} \usepackage{amssymb} \usepackage{amsopn}
\usepackage{amsthm} \usepackage{amsfonts} \usepackage{mathrsfs}

\usepackage[headings]{fullpage}
\usepackage[bottom]{footmisc}
\usepackage[dvips]{graphicx} \usepackage[usenames]{color}
\usepackage[all]{xy}
\usepackage[OT2, T1]{fontenc}
\usepackage[utf8]{inputenc}
\usepackage{dsfont}
\usepackage{longtable}

\usepackage[pdfpagelabels, pdftex]{hyperref}

\hypersetup{
  pdftitle={},
  pdfauthor={},
  pdfsubject={},
  pdfkeywords={},
  colorlinks=true,
  breaklinks=true,
  bookmarksopen=true,
  bookmarksnumbered=true,
  pdfpagemode=UseOutlines,
  plainpages=false}

\DeclareMathOperator{\rH}{H}

\DeclareMathOperator{\rK}{K}

\DeclareMathOperator{\rM}{M}

\DeclareMathOperator{\rT}{T}

\DeclareMathOperator{\rb}{b}
\DeclareMathOperator{\rc}{c}

\DeclareMathOperator{\rh}{h}

\newcommand{\bC}{{\mathbb C}}

\newcommand{\bF}{{\mathbb F}}

\newcommand{\bH}{{\mathbb H}}

\newcommand{\bN}{{\mathbb N}}

\newcommand{\bP}{{\mathbb P}}
\newcommand{\bQ}{{\mathbb Q}}
\newcommand{\bR}{{\mathbb R}}

\newcommand{\bT}{{\mathbb T}}

\newcommand{\bZ}{{\mathbb Z}}

\newcommand{\cD}{{\mathscr D}}

\newcommand{\cM}{{\mathscr M}}

\newcommand{\cX}{{\mathscr X}}
\newcommand{\cY}{{\mathscr Y}}

\newcommand{\dO}{{\mathcal O}}

\newcommand{\fO}{{\mathfrak O}}

\newcommand{\fo}{{\mathfrak o}}

\DeclareSymbolFont{cyrletters}{OT2}{wncyr}{m}{n}
\DeclareMathSymbol{\Sha}{\mathalpha}{cyrletters}{"58}

\newcommand{\one}{\mathbf{1}}
\newcommand{\surj}{\twoheadrightarrow} 
\newcommand{\inj}{\hookrightarrow}

\DeclareMathOperator{\Hom}{Hom}

\DeclareMathOperator{\Aut}{Aut}

\DeclareMathOperator{\coker}{coker}

\DeclareMathOperator{\GL}{GL}
\DeclareMathOperator{\PGL}{PGL}
\DeclareMathOperator{\PSL}{PSL}
\DeclareMathOperator{\SL}{SL}

\newcommand{\tr}{{\rm tr}} 

\newcommand{\matzz}[4]{\left(
\begin{array}{cc} #1 & #2 \\ #3 & #4 \end{array} \right)}

\DeclareMathOperator{\Spec}{Spec}
\DeclareMathOperator{\Proj}{Proj}

\DeclareMathOperator{\Pic}{Pic}

\DeclareMathOperator{\Nrd}{\rm Nrd}
\DeclareMathOperator{\trd}{\rm trd}

\DeclareMathOperator{\Frob}{Frob}

\DeclareMathOperator{\ord}{ord}

\DeclareMathOperator{\Tot}{Tot}

\DeclareMathOperator{\res}{res}

\DeclareMathOperator{\Gal}{Gal}

\newcommand{\ph}{\varphi}

\newcommand{\et}{\text{\rm \'et}}

\newcommand{\ab}{{\rm ab}}

\DeclareMathOperator{\DV}{{\rm DV}}
\DeclareMathOperator{\Sym}{\rm Sym}

\newcommand{\mgen}{\delta}
\DeclareMathOperator{\Lk}{{\rm Lk}}
\DeclareMathOperator{\Spf}{{\rm Spf}}
\newcommand{\ket}{\text{\rm k\'et}}

\newcommand{\secondtau}{\tau'}
\newcommand{\secondK}{K'}
\newcommand{\secondD}{D'}
\newcommand{\secondS}{S'}
\newcommand{\secondGamma}{\Gamma'}

\newcommand{\ov}[1]{\mbox{${\overline{#1}}$}} 

\makeatletter

\newcommand{\Rmnum}[1]{\expandafter\@slowromancap\romannumeral #1@}
\makeatother

\newtheorem{thm}{Theorem}

\newtheorem{prop}[thm]{Proposition}
\newtheorem{lem}[thm]{Lemma}

\newtheorem{cor}[thm]{Corollary}

\newtheorem{thmABC}{Theorem}

\theoremstyle{definition}
\newtheorem{defi}[thm]{Definition}

\theoremstyle{remark}
\newtheorem{rmk}[thm]{Remark}

\newenvironment{pro*}[1][Proof]{{\it{#1:}} }{}
\newenvironment{pro**}[1][]{{\it{#1}} }{\hfill $\square$}

\numberwithin{equation}{section}

\usepackage{enum}

\newlist{enumer}{enumerate}{2}
\setlist[enumer]{label=(\roman*),align=left,labelindent=0pt,leftmargin=*,widest = (iii)}
\newlist{enumerar}{enumerate}{1}
\setlist[enumerar]{label=\arabic*.,align=left,labelindent=0pt,leftmargin=*,widest = 8.}
\newlist{enumera}{enumerate}{2}
\setlist[enumera]{label=(\arabic*),align=left,labelindent=0pt,leftmargin=*,widest = (8)}
\newlist{enumeral}{enumerate}{2}
\setlist[enumeral]{label=(\alph*),align=left,labelindent=0pt,leftmargin=*,widest = (m)}

\begin{document}

%
%
%
%
\hrule width\hsize
\hrule width\hsize
\hrule width\hsize

\vspace{1.2cm}

\title[\Kurztitel]{\Ueberschrift} 
\author{Jakob Stix}
\address{Jakob Stix, MATCH - Mathematisches  Institut, Universit\"at Heidelberg, Im Neuenheimer Feld 288, 69120 Heidelberg, Germany}
\email{stix@mathi.uni-heidelberg.de}
 \urladdr{http://www.mathi.uni-heidelberg.de/~stix/}

\author{Alina Vdovina}
\address{Alina Vdovina,  School of Mathematics and Statistics, Newcastle University, Newcastle upon Tyne, NE1 7RU, UK}
\email{alina.vdovina@ncl.ac.uk}
\urladdr{http://www.staff.ncl.ac.uk/alina.vdovina/}

\thanks{The authors acknowledge the hospitality and support provided by MATCH at Universit\"at Heidelberg, Newcastle University, and the Isaac Newton Institute in Cambridge.}
\date{\today} 

\maketitle

\begin{quotation} 
\noindent \small {\bf Abstract} --- We construct 
an infinite  series of simply transitive irreducible lattices in $\PGL_2(\bF_q((t))) \times \PGL_2(\bF_q((t)))$ by means of a quaternion algebra over $\bF_q(t)$. The lattices depend  on an odd prime power $q = p^r$ and a parameter $\tau \in \bF_q^\ast, \tau \not= 1$, and are the fundamental group of a square complex with just one vertex and universal covering $T_{q+1} \times T_{q+1}$, a product of trees with constant valency $q+1$. 

Our lattices give rise via non-archimedian uniformization to smooth projective surfaces of general type over $\bF_q((t))$ with ample canonical class, Chern ratio $\rc_1^2/\rc_2 = 2$, trivial Albanese variety and non-reduced Picard scheme. For $q = 3$, the Zariski-Euler characteristic attains its minimal value $\chi = 1$: the surface is a non-classical fake quadric.
\end{quotation} 


\setcounter{tocdepth}{1} {\scriptsize \tableofcontents}

\section{Introduction}
This paper explores the overlap of two different worlds: simply transitive lattices in products of trees and arithmetic lattices of rank $2$ in positive characteristic. 
The lattices that we construct share the properties of both worlds and as such are the first  of their kind: 
\begin{enumer}
\item they are arithmetic quaternionic lattices of rank $2$ in positive characteristic, and
\item their action on the product of two trees with the same valency in both factors is  simply transitive on  vertices, and 
\item we are able to describe explicitly a finite presentation of our lattices.
\end{enumer}

\smallskip

Lattices in products of trees provide examples for many interesting group theoretic properties,  for example there are finitely presented infinite simple groups \cite{burger-mozes:simple}, and many are not residually finite. If we restrict to torsion free lattices that act simply transitively on the vertices of the product of trees (not interchanging the factors), then those lattices are fundamental groups of square complexes with just one vertex, complete bipartite link and a VH-structure, see Section~\S\ref{sec:VHstructure}. This is combinatorially well understood and there are plenty of such lattices, see Section~\S\ref{sec:mass_formula} for a mass formula. 

Let $T_n$ denote the tree of constant valency $n$. For different odd prime numbers $p \not= \ell$, Mozes  \cite{mozes:cartan} \S3, 
Burger and Mozes \cite{burger-mozes:simple} and \cite{burger-mozes:lattices}~\S2.4 for $p$ and $\ell$ congruent to $1$ mod $4$, 
and later  in general Rattaggi  \cite{rattaggi:thesis}~\S3  found a lattice of arithmetic origin acting on $T_{p+1} \times T_{\ell +1}$ with  
simply transitive action on the vertices. This lattice is a $\{p,\ell\}$-arithmetic group for the Hamiltonian quaternions over $\bQ$.

\smallskip

Now we take the point of view from the arithmetic side. Let $D$ be a quaternion algebra over a global function field $K/\bF_q$ of a smooth curve over $\bF_q$, and let $S$ consist of the ramified places of $D$ together with two distinct unramified $\bF_q$-rational places $\tau$ and $\infty$. An $S$-arithmetic subgroup $\Gamma$ of the projective linear group $G = \PGL_{1,D}$ of $D$ acts as a cocompact lattice on the product of trees  $T_{q+1} \times T_{q+1}$ that are the Bruhat--Tits buildings for $G$ locally at $\tau$ and $\infty$. There are plenty of such arithmetic lattices, however, in general the action of $\Gamma$ on the set of vertices will not be simply transitive.

\smallskip

We now sketch the constuction of our lattices. 
Let $q$ be an odd prime power and let $c \in \bF_q^\ast$ be a non-square. Consider the $\bF_q[t]$-algebra with non-commuting variables $Z,F$
\[
\fO = \bF_q[t]\{Z,F\}/(Z^2 = c, F^2 = t(t-1), ZF = -FZ),
\]
an $\bF_q[t]$-order in the quaternion algebra $D = \fO \otimes \bF_q(t)$ over $\bF_q(t)$ ramified in $t=0,1$.
Let 
\begin{equation} \label{eq:firstdefinegroup}
G = \PGL_{1,\fO}
\end{equation}
be the algebraic group over $\bF_q[t]$ of units in $\fO$ modulo the center, a twisted form of $\PGL_2$, at least generically over 
$\bF_q[t,\frac{1}{t(t-1)}]$. 

\begin{thmABC}[see Theorem~\ref{thm:presentationLambda}] \label{thm:thmA}
Let $q$ be an odd prime power,  and choose a generator $\delta$ of the cyclic group $\bF_{q^2}^\ast$. Let 
$\tau \not= 1$ be an element of $\bF_q^\ast$, and $\zeta \in \bF_{q^2}^\ast$ be an element with 
norm $\zeta^{1+q} = (\tau-1)/\tau$. Let $G/\bF_q[t]$ be the algebraic group of \eqref{eq:firstdefinegroup}.

The irreducible arithmetic lattice $G(\bF_q[t,\frac{1}{t(t-\tau)}])$ has the following presentation: 
\begin{eqnarray*}
G(\bF_q[t,\frac{1}{t(t-\tau)}])  & \simeq & \left\langle \  d,a,b  \ \left| 
\begin{array}{c}
d^{q+1} =  a^2 = b^2 = 1 \\[0.3ex]
(d^i a d^{-i})(d^{j} b d^{-j}) = (d^l b d^{-l})(d^{k} a d^{-k})  \\[0.3ex]
\text{ for all } 0 \leq i,l \leq q \text{ and } j,k \text{ determined by  } (\star) 
\end{array} 
\right.
\right\rangle,
\end{eqnarray*}
where $(\star)$ is the system of equations in  the quotient group $\bF_{q^2}^\ast/\bF_q^\ast$
\[
\mgen^{j-l} =  (1 -  \zeta \mgen^{(i-l)(1-q)}) \cdot  \mgen^{(q+1)/2} \ \text{ and } \
\mgen^{k-i}  =  (1 -  \frac{1}{\zeta \mgen^{(i-l)(1-q)}}) \cdot \mgen^{(q+1)/2}.
\]
\end{thmABC}

We also establish a presentation of the arithmetic group $\Lambda_\tau = G(\bF_q[t,\frac{1}{t(t-1)(t-\tau)}])$. The group $\Lambda_\tau$ is a semidirect product of a dihedral group of order $2(q+1)$ with a normal subgroup 
$\Gamma_\tau$ 
which has the following properties:
\begin{enumer}
\item
$\Gamma_\tau$ is an irreducible torsion free arithmetic lattice of rank $2$ with explicit finite presentation described in Theorem~\ref{thm:presentationgamma}.
\item  
$\Gamma_\tau$ acts simply transitively on $T_{q+1} \times T_{q+1}$ as the product of the Bruhat--Tits trees at $\tau$ and $\infty$ as above (Theorem~\ref{thm:presentationgamma}).
\item 
$\Gamma_\tau$ is residually finite, in fact even residually pro-$p$ by Proposition~\ref{prop:resprop}.
\item
$\Gamma_\tau$ is a FAB-group (all finite index subgroups have finite abelianization), more precisely, all 
non-trivial normal subgroups of $\Gamma_\tau$ have finite index (Proposition~\ref{prop:justinfinite}).
\item 
$\Gamma_\tau$ has cohomological dimension $2$ being the fundamental group of the square complex with one vertex 
\[
S_{\Gamma_\tau} = \Gamma_\tau \backslash T_{q+1} \times T_{q+1}.
\]
In fact, the group $\Gamma_\tau$ has a finite classifying space $S_{\Gamma_\tau} = \rK(\Gamma_\tau,1)$ of dimension $2$, hence is a group of type $WFL$, 
since $S_{\Gamma_\tau}$ has a contractible universal covering space, see Corollary~\ref{cor:explicitclassifyingspace}. 
\end{enumer}

The lattices $\Gamma_\tau$ are remarkable in several ways: as arithmetic lattices they have the rare property of being torsion free and acting simply transitively on the product of trees, as a fundamental group of a one vertex square complex they are special by being of arithmetic origin.  Moreover, unlike previous known lattices sharing both properties, see \cite{mozes:cartan} \S3, \cite{burger-mozes:simple}  and \cite{rattaggi:thesis}~\S3, the action is on a product of trees of the same valency (the valency being an odd prime power plus $1$) and the arithmetic is not $p$-adic but purely of positive characteristic.

\smallskip

The elements $\delta$ and $\zeta$ are necessary for the description of $\Gamma_\tau$ but do not change the group; however, the element $\tau$ is a true parameter of the construction.

\begin{thmABC}[see Corollary~\ref{cor:isomclass}] \label{thm:thmB}
Lattices of the form $\Gamma_\tau$ are commensurable if and only if they are isomorphic. They are isomorphic if and only if suitable  Galois conjugates of the parameters $\tau$ agree or add up to $1$.
\end{thmABC}

The question to explicitly determine a presentation of an arithmetic group is an old one. It is theoretically known which 
arithmetic lattices (in reductive groups) admit a finite presentation, see \cite{behr}, and it has been algorithmically solved 
for Fuchsian groups (lattices in $\PSL_2(\bR)$) with an algorithm described in Voight \cite{voight:funddomainfuchsian}.
Presentations of $\SL_2(\fo_k)$ for $k$ imaginary quadratic (a double cover of Bianchi groups) were determined by 
Swan in \cite{swan:bianchigrouppresentation}~\S4, even explicitly for a few small discriminants.
B\"ockle and Butenuth \cite{boecklebutenuth}~\S6 describe an algorithm that computes a fundamental domain and thus a presentation of a quaternionic lattice in the rank $1$ case over global function fields, see also Gekeler and Nonnengardt \cite{gekelernonnengardt} for congruence subgroups of $\GL_2(\bF_q[t])$. The analogous isotropic case of $\PSL_2$ is treated in Serre \cite{serre:trees} II \S2 but also only in the rank $1$ case in terms of the geometry of trees (not finitely generated by Nagao's theorem).  All these explicit results are in the rank $1$ case. 

In contrast, for arithmetic lattices of rank $2$ like our $\Lambda_\tau$ the (explicit) results are limited so far. 
To some extent Hilbert modular groups $\PSL_2(\fo_k)$ for a real quadratic number field $k$ are 
(split) analogues of $\Lambda_\tau$. Even though finite presentations  of Hilbert modular groups follow from the work of 
Blumenthal \cite{blumenthal}, Maa\ss\ \cite{maass:hilbertmodulargroup} and Herrmann \cite{herrmann}, the combinatorial geometry describing these presentations are quite involved 
so that explicit finite presentations are only known for small discriminants, and an algorithm is given 
in \cite{kirchheimerwolfahrt} when the class number of the real quadratic field is $1$.  In \cite{papikian:GLn} Papikian computes the Betti numbers of the quotient simplicial complex for certain  quaternionic lattices in $\PGL_n(\bF_q((t)))$ and general $n$.

Theorem~\ref{thm:thmA} provides 
explicit presentations for a family of arithmetic quaternionic lattices of rank $2$ over positive characteristic function fields $\bF_q(t)$ depending on a geometric parameter $\tau$ and an odd prime power parameter $q$.

\smallskip

Here is an application of our lattices $\Gamma_\tau$ to the theory of algebraic surfaces in characteristic~$p$. 
In \cite{mumford:fakep2} Mumford constructed the first fake $\bP^2$ by means of $p$-adic uniformization based on a torsion free lattice that acts simply transitively on the building of $\PGL_3(\bQ_2)$ with quotient complex of Euler characteristic $1$. 
Our lattice $\Gamma$ acts simply transitively on the building of $\PGL_2\big(\bF_q((y))\big) \times \PGL_2\big(\bF_q((t))\big)$, and for $q = 3$ (and necessarily $\tau = 2$) allows the construction of a non-classical fake quadric in characteristic~$3$, because the quotient square complex has Euler characteristic $\chi = (q-1)^2/4 = 1$. 

Fake quadrics are minimal surfaces of general type $S$ that have the same numerical invariants $\rc_1^2 =8$, $\rc_2 = 4$, and $\rh^1(S,\dO_S) = 0$ as the quadric $\{XY = Z^2\} \simeq \bP^1 \times \bP^1$ in $\bP^3$. The first examples of fake quadrics have been constructed by Kuga and his student Shavel \cite{shavel:fakequadric} by means of complex uniformisation using quaternionic lattices in $\PSL_2(\bR) \times \PSL_2(\bR)$. The construction here is a characteristic $p$ analogue, but only leads to a non-classical fake quadric, because the condition $\rh^1(S,\dO_S) = 0$ must be replaced by the property "$\Pic^0$ is non-reduced of dimension $0$" (hence the Albanese variety is nevertheless trivial).

Let $\ov{\pi}_1^\et(-)$ denote the \'etale fundamental group of a variety base changed to the algebraic closure of its field of definition, and let $\widehat{\Gamma}$ denote the pro-finite completion of a group $\Gamma$.

\begin{thmABC}[see Theorem~\ref{thm:ncfakequadric}] \label{thm:thmC}
Let $\Gamma_\tau$ be one of the lattices described above. There is a smooth projective surface of general type $X_\tau$ over 
$\bF_q((t))$ with ample canonical bundle, Chern ratio $\rc_1^2/\rc_2 = 2$, trivial Albanese variety, non-reduced Picard scheme and geometric \'etale fundamental group with an infinite continuous quotient
\[
\ov{\pi}_1^{\et}(X_\tau) \surj \widehat{\Gamma}_\tau.
\]
For $q = 3$, the surface $X_\tau$ is a non-classical fake quadric over $\bF_3((t))$.
\end{thmABC}

\subsection{Outline of paper} In Section~\S\ref{sec:groupsontrees} we recall well known facts on the geometry of square complexes, define the notion of a VH-structure in a group, and prove a criterion (Proposition~\ref{prop:mainabstractBMgroupresult}) that detects if a VH-structure in a group is actually "universal".  

Section~\S\ref{sec:quatlat} contains the arithmetic construction of the quaternionic lattice $\Gamma_\tau$ culminating in Proposition~\ref{prop:establishVHstructure} that establishes a VH-structure in $\Gamma_\tau$. 
In Section~\S\ref{sec:ggt} we determine vertex stabilizers and the geometry of the action of the arithmetic lattice $\Gamma_\tau$ on the product of Bruhat--Tits trees, we prove Theorem~\ref{thm:thmA} and determine presentations of all arithmetic lattices dealt with in this paper.

In Section~\S\ref{sec:finitegrouptheory} we determine the local structure of the quotient square complex $\Gamma_\tau \backslash T_{q+1} \times T_{q+1}$, describe a certain abelian quotient of $\Gamma_\tau$ and show that $\Gamma_\tau$ is residually pro-$p$.
Section~\S\ref{sec:rigidity} contains classification results up to commensurability and up to isomorphism. Here we 
prove Theorem~\ref{thm:thmB}. 

Section~\S\ref{sec:fakequadric} is devoted to the application to algebraic surfaces in positive characteristic. We use the simply 
transitive lattices for the construction of non-classical fake quadrics by non-archimedian uniformization thereby proving Theorem~\ref{thm:thmC}.

\subsection{Notation, terminology} 
By a lattice we mean a discrete group that acts with finite covolume on a space, here a product of trees, 
and thus finite covolume means with compact quotient space. 
The action of an element $g$ or a set of group elements $A$ on $v$ will be denoted by $g.v$ or by $A.v$ respectively.

The following table lists non-standard notation (sorted by order of appearance) which is kept in use throughout more than one section:

\begin{longtable}{l@{\quad}l}
$V(S)$ & set of vertices of a square complex $S$ \\
$E(S)$ & set of oriented edges of a square complex $S$ \\
$S(S)$ & set of squares of a square complex $S$ \\
$\ov{E}(S)$ & set of unoriented edges of a square complex $S$ \\
$T_n$ & infinite tree of constant valency $n$ \\
$S_{A,B}$ & square complex associated to VH-structure $A,B$ in a group \\
$p$ & an odd prime number \\
$q$ & a power of $p$ \\
$\tau$ & parameter of the construction $\tau  \in \bP^1(\bF_q)$ different from $0,1,\infty$ \\
$c$ & a non-zero parameter in $\bF_q \setminus \bF_q^2$ \\
$K$ & the function field $\bF_q(t)$ in one variable $t$ \\
$D$ & quaternion algebra over $K$ with invariants $c$ and $t(t-1)$ \\
$Z,F$ & anticommuting generators of $D$ with $Z^2 = c$ and $F^2 = t(t-1)$ \\
$\bF_q(z)$ & quadratic extension of $K = \bF_q(t)$ splitting $D$ such that $\bF_q((z)) = \bF_q((t^{-1}))$ \\
$\rM_2(-)$ & the $2$ by $2$ matrix algebra over a ring \\
$\rho_z$ & explicit choice of splitting  $D \inj \rM_2(\bF_q(z))$ \\
$G$ & the group $G=\PGL_{1,D}$ as an algebraic group over $K$ \\
$R' \subset R$ & explicit Dedekind rings with fraction field $K$ \\
$\fO' \subset \fO$ & explicit maximal $R'$ (resp. \ $R$) orders in $D$ \\
$\bF_q[Z]$ & subfield in $D$, quadratic over $\bF_q$\\
$\mgen$ & fixed generator of the multiplicative group $\bF_q[Z]^\ast$ \\
$\Delta$ & dihedral group of order $2(q+1)$ in $G(K)$ \\
$d,s$ & rotation and reflection in $\Delta$: images of $\delta$ and $F$ \\
$\wp$ & the map $\wp(w) = w/\bar w$ \\
$N$ & norm map $N: \bF_q[Z]^\ast \to \bF_q^\ast$ \\
$\bT$ & the torus of norm $1$ for the extension $\bF_q[Z]/\bF_q$ as an algebraic group over $\bF_q$ \\
$\bT_u$ & the smooth projective conic that completes the $\bT$-torsor of elements of norm $u$ \\
$\sigma_\xi(w)$ & automorphism $\sigma_\xi(w)= \bar w \cdot \wp(w - \xi)$ of $\bT_u$  \\
$\zeta$ & fixed choice of an element in $\bF_q[Z]$ of norm $N(\zeta) = \frac{\tau-1}{\tau}$ \\ 
$\bF_q(y)$ & quadratic extension of $K = \bF_q(t)$ splitting $D$ such that $\bF_q((y)) = \bF_q((t-\tau))$ \\
$\rho_y$ & explicit choice of splitting  $D \inj \rM_2(\bF_q(y))$ \\
$\gamma_\xi$ & quaternion $\gamma_\xi = tZ + \xi F$ \\
$N_c$ & the set of $\xi \in \bF_q[Z]$ of norm $N(\xi) = -c$ \\
$M_\tau$ & the set of $\eta \in \bF_q[Z]$ of norm $N(\eta) = \frac{c\tau}{1-\tau}$ \\
$\alpha_\xi$ & quaternion $\alpha_\xi  = (tZ + \xi F)Z$ \\
$a_\xi$ & image of $\alpha_\xi$ in $G(R')$ \\
$\beta_\eta$ & quaternion $\beta_\eta = (tZ + \eta F)Z$ \\
$b_\eta$ & image of $\beta_\eta$ in $G(R')$ \\
$A=A_\tau$ & the set of all $a_\xi$ for $\xi \in N_c$ \\
$B = B_\tau$ & the set of all $b_\eta$ for $\eta \in M_\tau$ \\
$\alpha$ & the element $\gamma_Z = tZ + ZF$ \\
$\beta$ & the element $\gamma_{Z/\zeta} = tZ + \zeta^{-1} ZF$ \\
$\Gamma = \Gamma_\tau $ &   the group generated by the sets $A$ and $B$ in $G(K)$ \\ 
$\Lambda' \subseteq \Lambda$ & the subgroups generated by $d,a,b$ respectively $d,s,a,b$ in $G(K)$ \\
$\Sym(M)$ & Symmetric group of the set $M$ \\
$P_A$, $P_B$ & local permutation groups for a VH-structure $A,B$ in a group \\
$S_{\Gamma_\tau}$ & the one vertex square complex $\Gamma_\tau \backslash T_{q+1} \times T_{q+1}$, a finite classifying space for $\Gamma_\tau$ 
\end{longtable}

\section{Preliminaries on groups acting on products of trees} \label{sec:groupsontrees}

We give a quick introduction to the geometry of square complexes and fix the terminology.

\subsection{Square complexes}
Recall that a square complex $S$ consists of a graph $S^1 = (V(S),E(S))$ with set of vertices $V(S)$, and set of oriented edges $E(S)$, together with a set of squares $S(S)$ that are combinatorially attached to the graph $S^1$ as explained below. Reversing the orientation of an edge $e \in E(S)$ is written as $e \mapsto e^{-1}$ and the set of unoriented edges is the quotient set 
\[
\ov{E}(S) = E(S)/(e \sim e^{-1}).
\]

More precisely, a square $\square$ is described by a $4$-tuple of oriented edges $e_i \in E(S)$
\[
\square = (e_1,e_2,e_3,e_4)
\]
where  the origin of $e_{i+1}$ is the terminus of $e_i$ (with $i$ modulo $4$). Such $4$-tuples describe the same square if and only if they belong to the same orbit under the dihedral action  generated by cyclically permuting the edges $e_i$ and by  the reverse orientation  map 
\[
(e_1,e_2,e_3,e_4) \mapsto (e_4^{-1},e_3^{-1},e_2^{-1},e_1^{-1}).
\]
We visualize squares $(e_1,e_2,e_3,e_4)$ as 
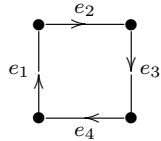
\[
\xymatrix@M0ex@R-2ex@C-2ex{
\bullet \ar[r]^(1){e_2} & \ar@{-}[r] & \bullet \ar[d]^(1){e_3} \\
\ar@{-}[u] & & \ar@{-}[d] \\
\bullet  \ar[u]^(1){e_1} & \ar@{-}[l] & \bullet \ar[l]^(1){e_4} \\
}
\]
For $x \in V(S)$ let $E(x)$ denote the set of oriented edges originating in the vertex $x$. For more details on square complexes we refer for example to \cite{burger-mozes:lattices} \S1.

\subsubsection{Products of trees} 
Let $T_n$ denote the $n$-valent tree. The product of trees 
\[
M = T_{m} \times T_{n}
\]
is a Euclidean building of rank $2$ and a square complex.  In this note we are interested in  lattices, i.e., groups $\Gamma$ acting discretely and cocompactly on $M$ respecting the structure as square complex.  The quotient $S = \Gamma \backslash M$ is a finite square complex, typically with orbifold structure coming from the stabilizers of cells.

\subsubsection{Torsion free lattices} 
We are especially interested in the case where $\Gamma$ is torsion free and acts simply transitively on the set of vertices of $M$. These yield the smallest quotients $S$ without non-trivial orbifold structure. Since $M$ is a CAT(0) space, any finite group stabilizes a cell of $M$. Moreover, the stabilizer of a cell is pro-finite, hence compact, so that a discrete group $\Gamma$ acts with trivial stabilizers on $M$ if and only if $\Gamma$ is torsion free.

\subsubsection{Link}
The \textbf{link} at a vertex $x$ in a square complex $S$ is the (undirected multi-)graph $\Lk_x$ whose set of vertices is $E(x)$ and edges in $\Lk_x$ are squares in $S$ containing the respective edges of $S$,  see \cite{burger-mozes:lattices} \S1. 

\begin{prop}
The universal cover of a finite square complex $S$ is a product of trees if and only if the link $\Lk_x$ at each vertex $x \in S$ is a complete bipartite graph.
\end{prop}
\begin{proof}
This is well known and follows for example from \cite{ballmannbrin} Theorem C.
\end{proof}

\subsubsection{VH-structure}  \label{sec:VHstructure}
A \textbf{vertical/horizontal structure}, in short a \textbf{VH-structure}, on a square complex $S$ consists on a bipartite structure $\ov{E}(S) = E_v \sqcup E_h$ on the set of unoriented edges of $S$ such that for every vertex $x \in S$ the link $\Lk_x$ at $x$ is the complete bipartite graph on the induced partition of $E(x) = E(x)_v \sqcup E(x)_h$. Edges in $E_v$ (resp.\ in $E_h$) are called vertical (resp.\ horizontal) edges. See \cite{wise:thesis} for general facts on VH-structures. The \textbf{partition size} of the VH-structure is the function  $V(S) \to 
\bN \times \bN$ on the set of vertices 
\[
x \mapsto (\# E(x)_v, \# E(x)_h)
\]
or just the corresponding tuple of integers if the function is constant. Here $\#(-)$ denotes the cardinality of a finite set. 

We record the following basic fact, see \cite{burger-mozes:lattices} after Proposition 1.1:

\begin{prop} \label{prop:VHstructureanduniversalcover}
Let $S$ be a square complex. The following are equivalent.
\begin{enumeral}
\item 
The universal cover of $S$ is a product of trees $T_m \times T_n$ and the group of covering transformations does not interchange the factors.
\item 
There is a VH-structure on $S$ of constant partition size $(m,n)$. \hfill $\square$
\end{enumeral}
\end{prop}

\begin{cor} \label{cor:latticesversusVH}
Torsion free cocompact lattices $\Gamma$ in  $\Aut(T_m) \times \Aut(T_n)$ not interchanging the factors and 
up to conjugation correspond uniquely to finite square complexes with a VH-structure of partition size $(m,n)$ up to isomorphism.
\end{cor}
\begin{proof}
A lattice $\Gamma$ yields a finite square complex $S = \Gamma \backslash T_m \times T_n$ of the desired type. Conversely, a finite square complex $S$ with VH-structure of constant partition size $(m,n)$ has universal covering space $M = T_m \times T_n$ by Proposition~\ref{prop:VHstructureanduniversalcover}, and the choice of a base point vertex $\tilde{x} \in M$ above the vertex $x \in S$  identifies $\pi_1(S,x)$ with the lattice $\Gamma = \Aut(M/S) \subseteq \Aut(T_m) \times \Aut(T_n)$. The lattice depends on the  chosen base points only up to conjugation.
\end{proof}

Simply transitive torsion free lattices not interchanging the factors as in Corollary~\ref{cor:latticesversusVH} correspond to square complexes with only one vertex and a VH-structure, necessarily of constant partition size. These will be studied in the next section.

\subsection{One vertex square complexes}

Let $S$ be a square complex with just one vertex $x \in S$ and a VH-structure $\ov{E}(S) = E_v \sqcup E_h$. Passing from the origin to the terminus of an oriented edge induces a fixed point free involution on $E(x)_v$ and on $E(x)_h$. Therefore the partition size is necessarily a tuple of even integers.

\begin{defi} \label{defi:BMstructureingroup}
A \textbf{vertical/horizontal structure}, in short \textbf{VH-structure},  in a group $G$ is an ordered pair $(A,B)$ of finite subsets $A,B \subseteq G$ such that the following holds.
\begin{enumer}
\item \label{defiitem:AandBinvolution} Taking inverses induces fixed point free involutions on $A$ and $B$.
\item The union $A \cup B$ generates $G$.
\item \label{defiitem:ABequalsBA} The product sets $AB$ and $BA$ have size $\#A \cdot \#B$ and $AB = BA$.
\item \label{defiitem:AB2torsion} The sets $AB$ and $BA$ do not contain $2$-torsion.
\end{enumer}
The tuple $(\#A,\#B)$ is called the \textbf{valency vector} of the VH-structure in $G$.
\end{defi}

\subsubsection{Construction} 
Similar to the construction in  \cite{burger-mozes:lattices} \S6.1 starting from a VH-datum, the following construction 
\begin{equation} \label{eq:constructionSAB}
(A,B) \leadsto S_{A,B}
\end{equation}
yields a square complex $S_{A,B}$ with one vertex and VH-structure starting from a VH-structure $(A,B)$ in a group $G$. The vertex set $V(S_{A,B})$ contains just one vertex $x$. The set of oriented edges of $S_{A,B}$ is the disjoint union 
\[
E(S_{A,B}) = A \sqcup B
\]
with the orientation reversion map given by $e \mapsto e^{-1}$. Since $A$ and $B$ are preserved under taking inverses, there is a natural vertical/horizontal structure such that $E(x)_h = A$ and $E(x)_v = B$. 

The squares of $S_{A,B}$ are constructed as follows. Every relation in $G$ of the form 
\begin{equation} \label{eq:relationabba}
ab = b'a'
\end{equation}
with $a,a' \in A$ and $b,b' \in B$ (not necessarily distinct) gives rise to a square 
\[
\square = (a,b,a'^{-1},b'^{-1}).
\] 
The following relations are equivalent to \eqref{eq:relationabba} and also yield the same square:
\begin{eqnarray*}
a'b^{-1} & = &  b'^{-1}a, \\
a^{-1}b' & = & ba'^{-1}, \\
a'^{-1}b'^{-1} & = & b^{-1}a^{-1}.
\end{eqnarray*}
Of these four squares the square complex $S_{A,B}$ contains only one. 

\begin{lem}
The link $\Lk_x$ of $S_{A,B}$ in $x$ is the complete bipartite graph $L_{A,B}$ with vertical vertices labelled by $A$ and horizontal vertices labelled by $B$.
\end{lem}
\begin{proof}
By \ref{defiitem:ABequalsBA} of Definition~\ref{defi:BMstructureingroup} every pair $(a,b) \in A \times B$ occurs on the left hand side  in a relation of the form \eqref{eq:relationabba} and therefore the link $\Lk_x$ contains $L_{A,B}$.

If \eqref{eq:relationabba} holds, then the set of left hand sides of equivalent relations 
\[
\{ab, a'b^{-1},a^{-1}b',a'^{-1}b'^{-1}\}
\]
is a set of cardinality $4$, because $A$ and $B$ and $AB$ do not contain $2$-torsion by Definition~\ref{defi:BMstructureingroup} \ref{defiitem:AandBinvolution} + \ref{defiitem:AB2torsion} and the right hand sides of the equations are unique by 
Definition~\ref{defi:BMstructureingroup} \ref{defiitem:ABequalsBA}. Therefore $S_{A,B}$ only contains $(\#A \cdot \#B)/4$ squares. It follows that $\Lk_x$ has at most as many edges as $L_{A,B}$, and, since it contains $L_{A,B}$, must agree with it.
\end{proof}

\begin{rmk}
Property \ref{defiitem:AB2torsion} in Definition~\ref{defi:BMstructureingroup} avoids the following example that lead
in the construction \eqref{eq:constructionSAB} above to degenerate squares in $S_{A,B}$.
Let $G$ be the group with presentation
\[
G = \langle a,b \ | \ (ab)^2 = (a^{-1}b)^2 = 1 \rangle
\]
which is isomorphic to the iterated semi-direct product
\[
(\bZ \rtimes \bZ) \rtimes \bZ/2\bZ = (\langle a^2 \rangle \rtimes \langle b \rangle) \rtimes \langle ab \rangle
\]
due to $ba^2b^{-1} = a^{-2}$, and $(ab)a^2 (ab)^{-1} = a^{-2}$ and $(ab)b(ab)^{-1} = b^{-1}a^{-2}$. 
It is easy to check that the sets $A = \{a,a^{-1}\}$ and $B = \{b,b^{-1}\}$ satisfy all properties of 
Definition~\ref{defi:BMstructureingroup} except \eqref{eq:constructionSAB}. However, the corresponding square in $S_{A,B}$ has the shape
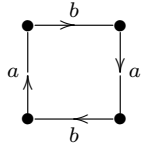
\[
\xymatrix@M0ex@R-2ex@C-2ex{
\bullet \ar[r]^(1){b} & \ar@{-}[r] & \bullet \ar[d]^(1){a} \\
\ar@{-}[u] & & \ar@{-}[d] \\
\bullet  \ar[u]^(1){a} & \ar@{-}[l] & \bullet \ar[l]^(1){b} \\
}
\]
and the link $\Lk_x$ is not complete bipartite containing double edges of the form:
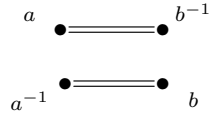
\[
\xymatrix@M0.1ex@R-2ex@C-2ex{
\bullet \ar@{=}[rr]^(-0.3){a}^(1.3){b^{-1}}  & & \bullet \\
 \  \bullet  \ar@{=}[rr]_(-0.3){a^{-1}}_(1.3){b} & & \bullet
}
\]
\end{rmk}

\subsubsection{Universal VH-structure} The presentation of certain arithmetic lattices will be determined
in Section~\S\ref{sec:determinepresentations} based on the following proposition which states a criterion for a VH-structure in a group to be universal among all VH-structures in groups sharing the same one vertex square complex.

\begin{prop} \label{prop:mainabstractBMgroupresult}
Let $G$ be a group with a VH-structure $A,B$ in $G$ and a cellular action of $G$  on $M = T_{2m} \times T_{2n}$ such that for a distinguished vertex $v \in M$ the orbit $A.v$ (resp.\ $B.v$) agrees with the set of vertical (resp.\ horizontal) neighbours of $v$. Then the following holds.
\begin{enumera}
\item \label{propitem:abstracttransitiveaction}
$G$ acts transitively on vertices of $M$.
\item 
If the valency vector of the VH-structure is $(\#A ,\#B) = (2m,2n)$, then 
\begin{enumer}
\item  \label{propitem:simplytransitive}
the action of $G$ on the vertices of $M$ is free and simply transitive, 
\item \label{propitem:squarecomplexes}
there is a canonical isomorphism of  the quotient square complex 
\[
G \backslash M \simeq S_{A,B}
\]
with the square complex $S_{A,B}$ of construction \eqref{eq:constructionSAB}, and 
\item \label{propitem:presentation}
with the vertex $\bar v \in S_{A,B}$, the group $G \simeq  \pi_1(S_{A,B},\bar v)$ has the presentation
\[
\pi_1(S_{A,B},\bar v)  = \left\langle x_a, x_b \text{ for } a \in A, b \in B  
\left| 
\begin{array}{c}
x_a x_{a^{-1}} =  x_b x_{b^{-1}} = 1 \text{ for } a \in A, b \in B \\
x_ax_b = x_{b'}x_{a'} \text{ for } a,a' \in A, \text{ and } b,b' \in B \\
\text{ such that }  ab = b'a' \text{ in G}
\end{array}
\right.
\right\rangle
\]
mapping $x_a \mapsto a$ and $x_b \mapsto b$.
\end{enumer}
\end{enumera}
\end{prop}
\begin{proof}
Since $v$ has all its neighbours in its orbit $G.v$, same same holds by structure transport via conjugation for all vertices in $G.v$. 
The only subsets of the set of vertices closed under taking neighbours are $M$ itself and the empty set because $M$ is connected. This proves assertion (1). 

\smallskip

We now prove assertion (2). Consider the square complex $S_{A,B}$ of construction \eqref{eq:constructionSAB} with one vertex and a VH-structure. By Proposition~\ref{prop:VHstructureanduniversalcover} its universal cover is isomorphic to 
\[
M' = T_{2m} \times T_{2n}
\]
(we use a different notation to distinguish the two copies of products of trees). The fundamental group 
$\pi_1(S_{A,B},\bar v)$ is generated by elements $x_a$ and $x_b$ for $a \in A$ and $b \in B$ representing loops along the corresponding oriented edges, and has the given presentation by construction, since squares in $S_{A,B}$ correspond to relations of the form $ab = b'a'$ in $G$. It follows that the maps $x_a \mapsto a$ and $x_b \mapsto b$ induce a well defined homomorphism
\[
\ph : \pi_1(S_{A,B},\bar v)   \surj G
\]
that is surjective because $A \cup B$ generates $G$. 
In this way $\pi_1(S_{A,B},\bar v)$ acts twice on a product of trees, namely on $M$  via $G$ and on $M'$ via identification with $\Aut(M'/S_{A,B})$ by means of a distinguished vertex $v' \in M'$. 

The map $v' \mapsto v \in M$ extends uniquely to a $\pi_1(S_{A,B},\bar v)$-equivariant map $f: M' \to M$ as follows. The action on $M'$ is free and transitive on vertices, so   the definition of $f$  is clear on the $0$-skeleton. The definition on the $1$-skeleton is by translating to an edge linking $v'$ with a neighbour and is consistent by the assumption on how $A,B \subset G$ moves $v$ to its neighbours. Extending $f$  to the $2$-skeleton, hence all of $M'$, works because $G$ has the same relations of length $4$ corresponding to squares at $v'$ resp.\ $v$. 
The induced map 
\[
f: M' \to M
\]
must be a covering map since this is true locally near $v' \mapsto v$ by construction and therefore by homogeneity everywhere. Because $M$ is simply connected and $M'$ is connected we conclude that $M ' \simeq M$ equivariantly by means of $f$. In particular, the map $\ph$ must be an isomorphism which shows  \ref{propitem:presentation}.  Furthermore, the $G$ action on $M$ is isomorphic via $(\ph,f)$ to the action on $M'$  of 
\[
\pi_1(S_{A,B},\bar v) \simeq \Aut(M'/S_{A,B})
\]
proving \ref{propitem:simplytransitive} and also inducing an isomorphism 
$G \backslash M \simeq S_{A,B}$
that verifies \ref{propitem:squarecomplexes}.
\end{proof}

\subsection{Mass formula for one vertex square complexes with VH-structure}  \label{sec:mass_formula}
Let $A$ (resp.\ $B$) be a set with fixed point free involution of size $2m$ (resp.\ $2n$). 
In order to count one vertex square complexes $S$ with VH-structure with vertical/horizontal partition $A \sqcup B$ of oriented edges  we introduce the generic matrix
\[
X = (x_{ab})_{a\in A, b \in B}
\]
with rows indexed by $A$ and columns indexed by  $B$ and with $(a,b)$-entry a formal variable $x_{ab}$. Let $X^t$ be the transpose of $X$, let $\tau_A$ (resp.\ $\tau_B$) be the permutation matrix for $e \mapsto e^{-1}$ for $A$ (resp.\ $B$). For a square $\square$  we set 
\[
x_\square = \prod_{e \in \square} x_e
\]
where the product ranges over the edges $e = (a,b)$ in the link of $S$ originating from $\square$ and $x_e = x_{ab}$. 
Then the sum of the $x_\square$, when $\square$ ranges over all possible squares with edges from $A \sqcup B$, reads 
\[
\sum_{\square} x_\square =  \frac{1}{4}\tr\big((\tau_A X \tau_B X^t\big)^2),
\]
and the number of  one vertex square complexes $S$ with VH-structure of partition size $(2m,2n)$ and edges labelled by $A$ and $B$  is given by 
\begin{equation} \label{eq:labeled_mass_formula}
\widetilde{{\rm BM}}_{m,n} = \frac{1}{(mn)!} \cdot \frac{\partial^{4mn}}{\prod_{a,b} \partial x_{ab} } \left( \frac{1}{4}\tr\big((\tau_A X \tau_B X^t\big)^2)\right)^{mn}. 
\end{equation}
Note that this is a constant polynomial.

We can turn this into a mass formula for the number of one vertex square complexes with VH-structure up to isomorphism where each square complex is counted with the inverse order of its group of automorphisms as its weight. We simply need to divide by the order of the universal relabelling 
\[
\#(\Aut(A,\tau_A) \times \Aut(B,\tau_B)) = 2^n(n)! \cdot 2^m(m)!.
\]
Hence the mass of one vertex square complexes with VH-structure is given by
\begin{equation} \label{eq:mass_formula}
{\rm BM}_{m,n} = \frac{1}{2^{n+m+2nm}(n)! \cdot (m)! \cdot (mn)!} \cdot \frac{\partial^{4mn}}{\prod_{a,b} \partial x_{ab} } \left( \tr\big((\tau_A X \tau_B X^t\big)^2)\right)^{mn}.
\end{equation}
Our formula \eqref{eq:labeled_mass_formula} reproduces the numerical values of $\widetilde{{\rm BM}}_{m,n}$ for small values $(2m,2n)$ that were computed by Rattaggi in \cite{rattaggi:thesis} table B.3. Here small means $mn \leq 10$.

\section{The quaternionic lattice} \label{sec:quatlat}

In this section we construct the quaternionic lattice $\Gamma$ in positive characteristic. Throughout the section $q$ will be a power of an odd prime $p > 2$. The construction is explicit and algebraic with some mild arithmetic input. It depends on a parameter 
\[
\tau  \in \bP^1(\bF_q) \setminus \{0,1,\infty\} = \bF_q^\ast \setminus \{1\}.
\]

\subsection{Enters the quaternion algebra} We fix once and for all a non-square
\[
c \in \bF_q^\ast \setminus (\bF_q^\ast)^2,
\]
and let $K = \bF_q(t)$ denote the rational function field over $\bF_q$ in the parameter $t$.
Let $D$ be the quaternion algebra over $K$ given as
\[
D =  \left(\frac{c,t(t-1)}{K}\right) 
\]
which in terms of non-commuting variables $Z,F$ has the presentation as an associative algebra
\[
D = K\{Z,F\}/(Z^2 = c, F^2 = t(t-1), ZF = - FZ).
\]
In fact, the quaternion algebra $D$ does not depend on the particular choice of the non-square $c$ up to isomorphism.

\begin{lem}
The  quaternion algebra $D$ ramifies exactly in $t = 0$ and $t = 1$.
\end{lem}
\begin{proof}
The algebra $D$ ramifies at most in $\{0,1,\infty\}$, the locus where the parameters $c$ and $t(t-1)$ are not both units.
The local Hilbert symbol at the place $v$ either $t=0$ or $t = 1$ computes as
\[
\big(c,t(t-1)\big)_v =  c^{(q+1)/2}  = -1,
\]
so  $D$ ramifies there. 
The number of ramified places being even,  $D$ is unramified at $t = \infty$. 
\end{proof}

\subsubsection{The first splitting of $D$}
In order to compute with localizations of $D$ later, we introduce a splitting of $D$.
Let $\bF_q(z)$ be the quadratic extension of $K = \bF_q(t)$ determined by 
\[
(z-1)^2  = \frac{t-1}{t}.
\]
The non-trivial Galois automorphism is given by $z \mapsto 2-z$ and
\[
t = \frac{1}{z(2-z)}.
\]

\begin{lem} \label{lem:split_z}
The quaternion algebra $D$ splits over $\bF_q(z)$, i.e., $D$ has a $2$-dimensional representation $\rho_z : D \to \rM_2(\bF_q(z))$ over $\bF_q(z)$ as follows:
\begin{eqnarray*}
Z &  \mapsto & \rho_z(Z) = \matzz{0}{c}{1}{0}, \\
F & \mapsto & \rho_z(F) =  \matzz{t(z-1)}{0}{0}{t(1-z)}.
\end{eqnarray*}
\end{lem}
\begin{proof}
This is an elementary computation in $\rM_2(\bF_q(z))$.
\end{proof}

\subsubsection{The algebraic group}
In the following, we will be concerned with arithmetic lattices for the algebraic group over $K$
\[
G = \PGL_{1,D}
\]
of the projective linear group of rank $1$ of $D$, a twisted form of the adjoint form $\PGL_2$ of $\SL_2$. 

\subsubsection{Choice of orders} Let us introduce the Dedekind ring
\[
R' = \bF_q[t,\frac{1}{t(t-\tau) }]
\]
of $S'$-integers of $\bF_q(t)$ where $S'$ is the set of places $\{0,\tau,\infty\}$,
note that $\tau \not= 0,1$. The quaternion algebra $D$ contains the $R'$-order
\[
\fO' = R' \oplus R' \cdot Z \oplus R' \cdot F \oplus R'\cdot ZF.
\] 
In this basis $1,Z,F,ZF$ the trace form has discriminant $16c^2 t^2 (t-1)^2$, so that the reduced discriminant of  $\fO'$ is the square-free ideal $(t(t-1))$ and $\fO'$ is a maximal $R'$-order of $D$. 
We further set 
\[
R = R'[\frac{1}{t-1}]
\]
and consider the scalar extension $\fO = \fO' \otimes_{R'} R$ which is even an Azumaya algebra over $R$, because both parameters 
$c, t(t-1) \in R^\ast$ are units in $R$ (and $p$ is odd). 

The choice of the basis $1,Z,F,ZF$ introduces naturally an integral structure on $G$ and thus arithmetic subgroups
\[
\fO'^\ast/R'^\ast = G(R') \ \subseteq \ \fO^\ast/R^\ast = G(R) \ \subseteq \ G(K) = \PGL_{1,D}(K) = D^\ast/K^\ast.
\]
The equalities follow because $R'$ and $R$ are principal ideal domains. 

\subsection{A constant quadratic subfield} Since $c$ is not a square in $\bF_q$, the element $Z$ generates a finite subfield of $q^2$ elements
\[
\bF_q[Z] \subset \fO'.
\]

\subsubsection{Dihedral group}
The multiplicative group of a finite field is cyclic, and we choose a generator
\[
\mgen \in  \bF_q[Z]^\ast.
\]
The quotient group 
$
\bF_q[Z]^\ast/\bF_q^\ast
$
is cyclic of order $q+1$ with generator $d$ the image $\mgen$.  

Since $F \in \fO$ and $F(Z)F^{-1} = -Z$, the induced map  
\[
F(-)F^{-1} : \bF_p[Z] \to \bF_p[Z]
\]
describes the non-trivial Galois automorphism of the quadratic field extension $\bF_q[Z]/\bF_q$, hence the $q$-Frobenius automorphism
\begin{equation} \label{eq:frobenius}
F(\mgen_1 + \mgen_2 Z)F^{-1} = (\mgen_1 + \mgen_2 Z)^q = \mgen_1 - \mgen_2 Z
\end{equation}
for $\mgen_1$, $\mgen_2$ in $\bF_q$. Let $s$ be the image of $F$ in $G(R) = \fO^\ast/R^\ast$.

\begin{lem} \label{lem:dihedralgroup}
The subgroup  
\[
\Delta :=  \langle d, s \rangle \subset  G(R) 
\]
is a dihedral group of order $2(q+1)$ with rotation $d$ and reflection $s$.
\end{lem}
\begin{proof}
The relations $d^{q+1} = 1$ and $s^2 = 1$ are clear, and 
$s d s = s(d)s^{-1} = d^{q} = d^{-1}$ by \eqref{eq:frobenius}. 
\end{proof}

\subsubsection{Elements of norm $1$ and Hilbert's Theorem 90}
Let us introduce homogeneous coordinates $[W_0:W_1]$ on $\bP^1$ and set formally $W = W_0 + W_1 Z$. We parametrize the projective space $\bP(\bF_q[Z])$ of the $2$-dimensional $\bF_q$-vector space $\bF_q[Z]$ as
\begin{eqnarray*}
\bP^1(\bF_q) \quad  = & \bP(\bF_q[Z]) & = \quad \bF_q[Z]^\ast/\bF_q^\ast, \\ 
{[w_0: w_1]} \quad \leftrightarrow & w_0 + w_1Z & \leftrightarrow \quad w.
\end{eqnarray*}
The Galois conjugation automorphism $w \mapsto \bar w$ of $\bF_q[Z]^\ast$ is the linear map 
\[
W = [W_0:W_1] \mapsto \bar W = [W_0:-W_1].
\]
Let us further introduce the map 
\begin{eqnarray*}
\wp : \bF_q[Z]^\ast/\bF_q^\ast & \to &  \bF_q[Z]^\ast, \\  
w_0 + w_1Z  & \mapsto & \wp(w_0+w_1Z) =  \frac{w_0+w_1Z}{w_0 - w_1Z}, \notag
\end{eqnarray*}
or in short $\wp(w) = w/\bar w$, and let $w \mapsto N(w) = w \cdot \bar w = w^{1+q}$ be the norm map 
\[
N = N_{\bF_q[Z]/\bF_q} : \bF_q[Z]^\ast \to \bF_q^\ast.
\]

\begin{lem} \label{lem:hilb90}
We have a short exact sequence
\[
1 \to \bF_q[Z]^\ast/\bF_q^\ast \xrightarrow{\wp} \bF_q[Z]^\ast \xrightarrow{N} \bF_q^\ast \to 1.
\]
\end{lem}
\begin{proof}
Field extensions of finite fields are all Galois with cyclic Galois groups. So the assertion is an immediate consequence of Hilbert's Theorem 90 and counting of orders of cyclic groups. 
\end{proof}

The norm-$1$ torus $\bT$ of the quadratic field extension $\bF_q[Z]/\bF_q$ is the algebraic torus 
\[
\bT = \{U_0 + U_1Z \ ; \ U_0^2 - c U_1^2 = 1\}.
\]
Lemma~\ref{lem:hilb90} provides an explicit parametrization of its rational points
\[
\bF_q[Z]^\ast/\bF_q^\ast = \bP(\bF_q[Z]) \xrightarrow{\wp} \ker(N) = \bT(\bF_q).
\]

\subsubsection{Elements of fixed norm} \label{sec:fixednorm}

Let $u \in \bF_q^\ast$ be a fixed element. The projective variety
\[
\bT_u = \{[W_0:W_1:V] \ ; \ W_0^2 - c W_1^2 = u V^2\} \subseteq \bP^2_{\bF_q},
\]
is a smooth conic over $\bF_q$ and thus isomorphic to the projective line $\bP^1_{\bF_q}$. Because $c$ is not a square in $\bF_q$, the rational points of $\bT_u$ have $V \not=0$ and thus 
\[
\bT_u(\bF_q) = \{w = w_0 + w_1Z \in \bF_q[Z] \ ; \ N(w_0+w_1Z) = u\}
\]
describes the set of elements of $\bF_q[Z]$ of fixed norm $u$.

We set formally $W = W_0 + W_1 Z$ and $\bar W = W_0 - W_1Z$, so that the norm map becomes the multiplicative map 
\[
N(W) = W \cdot \bar W = W_0^2 - cW_1^2
\]
and 
\[
\bT_u = \{[W:V] \ ; \ N(W) = uV^2\}.
\]
Conjugation $w \mapsto \bar w$ reads in this description in algebraic coordinates
\[
 [W:V]  = [W_0:W_1:V] \mapsto [\bar W:V]  = [W_0:-W_1:V]
\]
and yields an automorphism of $\bT_u$ due to $N(W) = N(\bar W)$. 

\begin{lem}  \label{lem:sigma_xi}
For $u \in \bF_q^\ast$ and $\xi \in \bF_q[Z]$, the map 
\[
\sigma_\xi: \bT_u \to \bT_u
\]
given in the formal variable $W = W_0 + W_1 Z$ by 
\[
[W:V] \mapsto [\bar W \cdot \frac{W - \xi V}{\bar W - \bar{\xi} V} : V]
\]
defines an automorphism which on $\bF_q$-rational points $w = [w : 1]$ induces the map
\[
w \mapsto \bar{w} \cdot \wp(w-\xi)  = w^q (w - \xi)^{1-q} = \bar{w} \cdot \frac{w - \xi}{\bar{w} - \bar{\xi}}. %
\]
\end{lem}
\begin{proof}
The map is well defined because
\[
N\left(\bar W \cdot \frac{W - \xi V}{\bar W - \bar{\xi} V} \right) = N(W)
\]
and by 
\[
[\bar W \cdot \frac{W - \xi V}{\bar W - \bar{\xi} V} : V] = [\frac{uV - \xi \bar W}{\bar W - \bar{\xi} V} : 1]
\]
of degree $1$, hence indeed an automorphism of $\bT_u$. The rest is obvious.
\end{proof}

\subsubsection{Choice of  coordinates} Let $\mu \in \bF_q[Z]$ be an element with $N(\mu) = u$. Then the map 
\begin{equation} \label{eq:deficoordinateonTu}
w \mapsto \mu \wp(w),
\end{equation}
or in homogeneous coordinates
\[
[W_0 : W_1] \mapsto  [\mu \cdot \wp(W): 1] = [\mu \cdot W^2 : N(W)],
\]
defines an algebraic isomorphism 
\[
\bP^1 \simeq \underline{\bP(\bF_q[Z])} \simeq \bT_u.
\]
Here $\underline{\bP(\bF_q[Z])}$ denotes the projective space variety associated to the $\bF_q$-vector space 
$\bF_q[Z]$.  Associated with the choice of $\omega$ and the associated coordinate $w$ on $\bT_u$ is an isomorphism 
\[
\Aut_{\bF_q}(\bT_u) = \PGL(\bF_q[Z]).
\]

\begin{lem} \label{lem:detsigma_xi}
In the coordinate $W$ introduced in \eqref{eq:deficoordinateonTu} the map $\sigma_\xi$ of Lemma~\ref{lem:sigma_xi} has the form
\[
\sigma_\xi : W \mapsto W - \frac{\xi}{\mu}  \cdot \bar W
\]
and determinant modulo squares represented by 
\[
\det(\sigma_\xi) = 1 - N(\xi/\mu).
\]
\end{lem}
\begin{proof}
We compute in terms of the homogeneous coordinate $W$ on $\underline{\bP(\bF_q[Z])}$ with $\mu \wp(W)$ on $\bT_u$
\[
\sigma_\xi([\mu \cdot \wp(W) : 1]) = [ \bar \mu \cdot \bar W^2   \cdot \wp(\mu W^2 - \xi N(W))        : N(W) ] 
= [\mu \cdot \wp(W - \frac{\xi}{\mu} \bar W) : 1] 
\]
so that in terms of the coordinate $W$ the automorphism $\sigma_\xi$ has the claimed form.

We write $\xi/\mu = \omega_0 + \omega_1 Z$ and lift $\sigma_\xi$ to the linear map 
\[
w \mapsto w - \frac{\xi}{\mu} \bar w
\]
in $\GL(\bF_q[Z])$.  Its determinant is computed via the matrix with respect to the basis $1,Z$ as
\[
\det(\sigma_\xi) \equiv \det \matzz{1 - \omega_0}{-\omega_1}{c \cdot \omega_1}{1+\omega_0}
= 1 - (\omega_0^2 - c \omega_1^2) = 1 - N(\xi/\mu).
\]
This completes the proof.
\end{proof}

\subsubsection{Second splitting of $D$} We fix once and for all an element  $\zeta = \zeta_0 + \zeta_1 Z \in \bF_q[Z]$ such that 
\[
N(\zeta_0 + \zeta_1 Z) = \frac{\tau -1}{\tau}.
\]
Note that $\tau \not= 0,1$, so such a choice can be made by Lemma~\ref{lem:hilb90}.

Let $\bF_q(y)/\bF_q(t)$ be the quadratic extension  determined by 
\[
(y-1)^2 =  \frac{\tau}{\tau-1} \cdot \frac{t-1}{t}.
\]
The non-trivial Galois automorphism is again given by $y \mapsto 2-y$ and 
\[
t = \frac{\tau}{\tau - (\tau-1)(y-1)^2}.
\]
\begin{lem} \label{lem:split_y}
The quaternion algebra $D$ splits over $\bF_q(y)$, i.e., $D$ has a $2$-dimensional representation over $\bF_q(y)$ as follows:
\[
\rho_y : D \to \rM_2(\bF_q(y)),
\]
\begin{eqnarray*}
Z & \mapsto &  \rho_y(Z) = \matzz{0}{c} {1}{0}, \\                  
F & \mapsto &  \rho_y(F) = \matzz{\zeta_0 t (y-1)}{c\zeta_1 t (1-y)}{\zeta_1 t (y-1) }{\zeta_0 t (1-y)} .
\end{eqnarray*}
\end{lem}
\begin{proof}
This is an elementary computation in $\rM_2(\bF_q(y))$.
\end{proof}

\subsection{Enters the lattice} \label{sec:lattice}

Crucially for the whole construction are the properties of the following quaternions, for $\xi \in \bF_q[Z]$, 
\[
\gamma_\xi = t Z + \xi F
\]
 of reduced norm
\begin{equation} \label{eq:reducednormalphaxi}
\Nrd(\gamma_\xi) =   t \big(N(\xi) - (N(\xi)+c) t\big).
\end{equation}
Since $\gamma_\xi$ is purely imaginary, its square belongs to the center of $D$ and so its image in $G(K)$ is a $2$-torsion element.

\begin{defi} \label{defi:AB}
\begin{enumera}
\item We set $N_c = \bT_{-c}(\bF_q) = \{\xi \in \bF_q[Z]^\ast \ ; \ N(\xi) = -c\}$ and 
\[
M_\tau = \bT_{ \frac{c\tau}{1-\tau}}(\bF_q) = \{\eta \in \bF_q[Z]^\ast \ ; \ N(\eta) =  \frac{c\tau}{1-\tau}\}.
\]
\item
First, for $\xi \in N_c$ (resp.\ $\eta \in M_\tau$) we set 
\[
\alpha_\xi = \gamma_\xi Z \quad \text{(resp.\ } \beta_\eta = \gamma_\eta Z \text{)}
\]
and set $a_\xi$ (resp.\ $b_\eta$) for  its image in $G(K)$. Secondly, we define the sets
\begin{eqnarray} \label{eq:AB}
A  & = & \{a_\xi \ ; \ \xi \in N_c  \}, \\
B  & = & \{b_\eta \ ; \ \eta \in M_\tau  \}, \notag
\end{eqnarray}
considered as subsets of $G(K) = D^\ast/K^\ast$. If we want to emphasize the dependence on $\tau$ we write $A = A_\tau$ and $B = B_\tau$.
\end{enumera}
\end{defi}

By \eqref{eq:reducednormalphaxi}  and $\Nrd(Z) = -c$ the reduced norms
\begin{eqnarray} \label{eq:normab}
\Nrd(\alpha_\xi) & = & c^2 \cdot t , \\
\Nrd(\beta_\eta) & = &  \frac{c^2}{1-\tau} \cdot ( t - \tau) \cdot t  \notag
\end{eqnarray}
are units in $R'$, hence $\alpha_\xi, \beta_\eta \in \fO'^\ast$
and therefore 
\[
A,B \subset G(R').
\]

\begin{lem} \label{lem:sizeAB}
The sets $A$ and $B$ are disjoint and of size $q+1$.
\end{lem}
\begin{proof}
Since $c \not= 0$ we never have $c\tau/(1- \tau) = -c$ and thus $N_c \cap M_\tau = \emptyset$. Since both $N_c$ and $M_\tau$ are the sets of $\bF_q$-rational points of an algebraic curve isomorphic to $\bP^1$, the lemma follows if the map 
\[
\xi \mapsto \gamma_\xi
\]
is injective as a map $\bF_q[Z]^\ast \to D^\ast/K^\ast$. By assuming the contrary we have $f(t) \in K^\ast$ and $\xi \not=\eta$ in $\bF_q[Z]^\ast$ with 
\[
tZ + \xi F = f(t) \cdot (tZ + \eta F).
\]
This implies $f(t) = 1$ and thus a contradiction.
\end{proof}

\begin{prop} \label{prop:sactsAB}
The following describes the dihedral group $\Delta$ relative the sets $A$ and $B$:
\begin{enumera}
\item The set $A$ (resp.\ $B$) is closed under taking inverses.
\item 
The dihedral group $\Delta$ acts by conjugation on $A$ (resp.\ $B$) as on the regular $(q+1)$-gon. 
\item For  $\xi \in N_c$ and $\eta \in M_\tau$ 
\begin{eqnarray}  \label{eq:actionbysAB}
 s a_\xi s & = & a_{- \bar{\xi}}, \\
 s b_\eta s & = & b_{-\bar{\eta}}. \notag
\end{eqnarray}
\item
The centralizer of $a_\xi$ (resp.\ $b_\eta$) in $\Delta$ is the subgroup $\langle s d^r \rangle$ with $r \in \bZ/(q+1)\bZ$ determined by  $\wp(\xi \delta^r) = -1$ (resp.\ $\wp(\eta \delta^r) = -1$).
\end{enumera}
\end{prop}
\begin{proof}
(1) The sets $N_c$ and $M_\tau$ are closed under the map $w \mapsto - w$. We compute for $\xi \in N_c$ 
\[
\alpha_\xi \alpha_{-\xi} = (tZ + \xi F)Z(tZ - \xi F) Z = (tZ+\xi F)^2 \cdot Z^2 \in K^\ast
\]
and similarly $\beta_\eta \beta_{-\eta} \in K^\ast$ for $\eta \in M_\tau$. This proves assertion (1).

(2) For $\lambda \in \bF_q[Z]^\ast$ and $\xi \in \bF_q[Z]$ we have
\begin{equation} \label{eq:actiondAB}
\lambda(\gamma_\xi Z)\lambda^{-1} = \lambda(tZ + \xi F)Z \lambda^{-1} = (tZ + \xi \wp(\lambda) F)Z = \gamma_{\xi \wp(\lambda)} Z.
\end{equation}
Clearly the map $\xi \mapsto \xi \wp(\lambda)$ identifies $N_c$ (resp.\ $M_\tau$) as a principal homogeous set under $\bF_q[Z]^\ast/\bF_q^\ast$ and so $A$ (resp.\ $B$) forms an orbit under conjugation by the cyclic group $\langle d \rangle = \bF_q[Z]^\ast/\bF_q^\ast$. 
By Lemma~\ref{lem:sizeAB} this orbit has size $q+1$ and thus forms the vertices of a regular $(q+1)$-gon. It remains to show that $s \in \Delta$ acts as a reflection. Hence (3) implies (2).

(3) The sets $N_c$ and $M_\tau$ are closed under the map $w \mapsto - \bar{w}$ so that the claim \eqref{eq:actionbysAB} is well posed. Conjugation by $s$ lifts in $D^\ast$ to 
\[
F(\gamma_\xi Z)F^{-1} = F(tZ + \xi F)ZF^{-1} = (tZ - \bar{\xi}F)Z = \gamma_{-\bar{\xi}}Z
\]
which confirms \eqref{eq:actionbysAB}.

(4) Using (3) and \eqref{eq:actiondAB}  we compute in $G(R)$ that 
\[
sd^r (\gamma_\xi Z) (sd^r)^{-1} = s(\gamma_{\xi \wp(\delta^r)}Z)s = \gamma_{-\bar{\xi}/\wp(\delta^r)} Z 
\]
which agrees by Lemma~\ref{lem:sizeAB} with $\gamma_\xi Z$ in $G(R)$ if and only if 
\[
\xi = - \bar{\xi}/\wp(\delta^r),
\]
or equivalently $\wp(\xi \delta^r) = -1$. 
\end{proof}

\begin{rmk} 
Computations using SAGE with small values of $q$ and $\tau$ yields that the conjugacy classes of reflections stabilizing elements of $A$ or $B$ sometimes coincide and sometimes does not.
\end{rmk}

\begin{defi} \label{defi:lattice}
\begin{enumera}
\item Recall that $\zeta \in \bF_q[Z]$ has norm $(\tau-1)/\tau$. We have natural choice of elements in $N_c$ and $M_\tau$, namely
\[
Z \in N_c
\]
\[
Z/\zeta \in M_\tau
\]
due to $N(Z) = -Z^2 = -c$ and $N(Z/\zeta) = c\tau/(1-\tau)$.
We abbreviate 
\begin{eqnarray*}
\alpha & = & \gamma_{Z} = tZ+ZF, \\
\beta & = & \gamma_{Z/\zeta} = tZ + \zeta^{-1}ZF
\end{eqnarray*}
and furthermore $a$ (resp.\ $b$) for the image of $\alpha$ (resp.\ $\beta$) in $G(K)$.
\item
With the notation from above we define the following finitely generated subgroups of $G(K)$:
\begin{eqnarray*}
\Lambda  \ & = & \langle d,s,a,b \rangle  \ \subseteq \ G(R), \\
\Lambda' & = & \langle d,a,b \rangle \ \subseteq \ G(R'), \\
\Gamma  \ & = & \langle A, B \rangle \ \subseteq \ G(R').
\end{eqnarray*}
If we want to emphasize the dependence on $\tau$ we write $\Lambda = \Lambda_\tau$, $\Lambda'  = \Lambda'_\tau$ and $\Gamma = \Gamma_\tau$.
\end{enumera}
\end{defi}

\begin{rmk}
Clearly there are inclusions $\Gamma \subseteq \Lambda' \subseteq \Lambda$, and we will see later in Section~\S\ref{sec:ggt} that these groups are indeed $S$-arithmetic lattices for the group $G = \PGL_{1,D}$ over $K = \bF_q(t)$.
\end{rmk}

\begin{prop}  \label{prop:gammaoffiniteindexinlambda}
The group $\Gamma$ is a normal subgroup of finite index in $\Lambda$.
\end{prop}
\begin{proof}
Since $\Lambda$ is generated by 
$d$, $s$, $a_{Z} = aZ$ and $b_{Z/\zeta}=b Z$, we have $\Lambda = \Delta \cdot \Gamma$ and 
it suffices to study the action of $d$ and $s$ by conjugation on $A$ and $B$. The assertion then follows from 
Proposition~\ref{prop:sactsAB}.
\end{proof}

\subsection{Relations of length $4$ in the lattice} \label{sec:rel4}
In this section we establish that  $A,B \subset \Gamma$ determines a VH-structure in $\Gamma$ in the sense of 
Definition~\ref{defi:BMstructureingroup}. 
This result is reminiscent of Dickson's theory of prime factorization in integral Hamiltonian quaternions, 
see \cite{dickson:quaternions} Theorem 8. 

\begin{prop} \label{prop:establishVHstructure}
With the notation from above the following holds.
\begin{enumera}
\item \label{propitem:ABequalsBA} The sets of products $AB = B A$ agree and have cardinality $(q+1)^2$. 
\item \label{propitem:relation4} The equation
\[
a_\xi b_{\eta} = b_{\lambda} a_{\mu}
\]
has for  every $(\xi,\eta)  \in N_c \times M_\tau$ a unique solution $(\mu,\lambda) \in N_c \times M_\tau$ and conversely for  every $(\mu,\lambda) \in  N_c \times M_\tau$ a unique solution $(\xi,\eta)  \in  N_c \times M_\tau$.
\item \label{propitem:no2-torsion}  None of the elements $a_\xi b_\eta$ and  $b_\eta a_\xi$  for all $\xi \in N_c$ and $\eta \in M_\tau$ is $2$-torsion. 
\item \label{propitem:BMstructure} The sets $A,B \subset \Gamma$ form a VH-structure in $\Gamma$.
\end{enumera}
\end{prop}
\begin{proof}
Assertions \ref{propitem:ABequalsBA} follows directly from  assertion \ref{propitem:relation4}.
In order to prove \ref{propitem:relation4} we compute in $D$ rather than in $D^\ast/K^\ast$, namely
\begin{eqnarray}
\alpha_\xi \beta_\eta & = & (tZ + \xi F)Z(tZ + \eta F)Z =  c (tZ + \xi F)(tZ - \eta F)  \label{eq:axibeta} \\
& = & c^2t^2 - c \xi \eta^q t(t-1) - c (\xi + \eta) t ZF   \notag \\
\beta_\lambda \alpha_\mu & = & (tZ + \lambda F)Z(tZ + \mu F)Z =  c (tZ + \lambda F)(tZ - \mu F) \notag \\
& = & c^2t^2 - c \lambda \mu^q t(t-1) - c (\lambda + \mu) t ZF    \notag
\end{eqnarray}
Since $D = \bF_q[Z](t) \oplus \bF_q[Z](t) \cdot ZF$ the equation in $D^\ast$ 
\[
\alpha_\xi \beta_\eta  =  f(t) \cdot \beta_\lambda \alpha_\mu
\]
with $f(t) \in K^\ast$ can only hold (consider the $ZF$-component) with a constant 
\[
f = f(t) \in \bF_q[Z] \cap \bF_q(t) = \bF_q.
\]
Evaluating at $t= 1$ yields
\[
c^2  - c (\xi + \eta) ZF  = f \cdot \big(c^2 - c (\lambda + \mu) ZF \big)
\]
so that $f = 1$ is forced. Now comparison of coefficients yields the system of equations
\begin{eqnarray} \label{eq:comparecoefficientsABequalsBA}
\xi \eta^q& = & \lambda \mu^q, \\
\xi + \eta & = &  \lambda + \mu. \notag
\end{eqnarray}
Note that the sums in the second equation are non-zero since otherwise $0 = N(-\xi/\eta) = \frac{\tau-1}{\tau}$ contradicting $\tau \not= 0,1$.
Because of $\mu^{q+1} = c = \xi^{q+1}$, the first equation is equivalent to $(\eta/\xi)^q = \lambda/\mu$. So the second equation reads
\[
\mu (1+ \bar \eta/ \bar \xi) = \mu(1+\lambda/\mu) = \mu + \lambda = \xi + \eta
\]
and finally \eqref{eq:comparecoefficientsABequalsBA} is equivalent to 
\begin{eqnarray} \label{eq:lambdamu_from_xieta}
\lambda & = &   \bar{\eta} \cdot \wp(\xi+\eta) , \\ 
\mu & = &  \bar{\xi} \cdot \wp(\xi+ \eta) . \notag 
\end{eqnarray}
Similarly, \eqref{eq:comparecoefficientsABequalsBA} is equivalent to 
\begin{eqnarray} \label{eq:xietau_from_lambdam}
\xi & = &   \bar{\mu} \cdot \wp(\lambda + \mu) , \\
\eta & = &  \bar{\lambda}  \cdot  \wp(\lambda + \mu) . \notag 
\end{eqnarray}
Note that the constraints on the norms of $\lambda$, $\mu$ and of $\xi$ and $\eta$ are satisfied automatically. This completes the proof of \ref{propitem:relation4}. 

Assertion \ref{propitem:BMstructure} follows by definition from \ref{propitem:ABequalsBA} and  \ref{propitem:no2-torsion}.
In order to show \ref{propitem:no2-torsion} we write 
\[
\alpha_\xi \beta_\eta = w_0 + w_1 Z + w_2 F + w_3 ZF = w_0 + \vec{w}
\]
and read from \eqref{eq:axibeta} that the imaginary part $\vec{w} \not= 0$ does not vanish and, by evaluating at $t=1$, 
\[
2 w_0|_{t = 1} = \trd(\alpha_\xi \beta_\eta|_{t=1}) = \trd(c^2 - c(\xi+\eta)ZF) = 2c^2 \not= 0
\]
that $w_0 \not= 0$. We therefore compute the square as
\[
(\alpha_\xi \beta_\eta)^2 = (w_0 + \vec{w})^2 = w_0^2 - \Nrd(\vec{w}) + 2w_0 \vec{w} \notin K,
\]
so that $(a_\xi b_\eta)^2 \not= 1$.
The claim for $b_\eta a_\xi$ follows from $AB = BA$. This completes the proof.
\end{proof}

\begin{cor} \label{cor:localpermutationwork}
For  $(\xi, \lambda) \in N_c \times M_\tau$ and with the notation of Section~\S\ref{sec:fixednorm} the relation
\[
a_\xi b_{\sigma_\xi(\lambda)}  = b_\lambda a_{\sigma_\lambda(\xi)}
\]
holds in $\Gamma$.
\end{cor}
\begin{proof}
It suffices to check the equations \eqref{eq:comparecoefficientsABequalsBA}. 
The first equation holds due to
\[
\xi \eta^q = \xi \bar{\eta} = \xi \cdot \lambda \frac{\bar \lambda - \bar \xi}{\lambda - \xi} 
= \lambda \cdot \xi \frac{\bar \xi - \bar \lambda}{\xi - \lambda} = \lambda \bar{\mu} = \lambda \mu^q,
\]
and 
\[
\xi + \eta = \xi + \bar{\lambda} \cdot \frac{\lambda - \xi}{\bar \lambda - \bar \xi} = \lambda + (\xi - \lambda)\big( 1 +  \frac{\bar \lambda}{\bar \xi - \bar \lambda}\big)
= \lambda + \bar{\xi} \cdot \frac{\xi-\lambda}{\bar \xi - \bar \lambda}  = \lambda + \mu.
\]
verifies the second equation.
\end{proof}

\section{Geometry of the quaternionic lattice} \label{sec:ggt}
In this section we keep the notation from Section~\S\ref{sec:quatlat} and study the group $\Gamma$ via geometric group theory unfolding in its action on products of trees.

\subsection{Reminder on the Bruhat Tits action}
Let  $K_x$ denote the completion of $K$ at $t = x$ with valuation $\ord_x$ and uniformizer $\pi_x$. We denote the ring of integers $\{f \ ; \ \ord_x(f) \geq 0\}$ in the local field $K_x$ by $\fo_x$. Let us recall that the vertices in the Bruhat Tits tree $T_x$ for $\PGL_2(K_x)$ are homothety classes of $\fo_x$-lattices in $(K_x)^2$. The tree $T_x$ has constant valency equal to the norm  $N(x)$ which equals the cardinality of the residue field $\kappa(x)$ at $t= x$. The group $\PGL_2(K_x)$ acts on $T_x$ simplicially by representatives in $\GL_2(K_x)$.
\[
\PGL_2(K_x) \to \Aut(T_x)
\]
We denote the vertex corresponding to the standard lattice $(\fo_x)^2$ with both coordinates in $\fo_x$ by $v_x \in T_x$ and call it the \textbf{standard vertex}. The stabilizer of $v_x$ is the group 
\[
\PGL_2(\fo_x)
\]
of elements that can be lifted to $\GL_2(\fo_x)$. 

\begin{lem} \label{lem:BTneighbours}
There are canonical bijections of the following three sets:
\begin{enumeral}
\item
The set of neighbours of $v_x \in T_x$.
\item $\{ M \in \GL_2(K_x) \cap \rM_2(\fo_x) \ ; \ \ord_x(\det(M)) = 1\}/\GL_2(\fo_x)$.
\item $\bP^1(\kappa(x))$.
\end{enumeral}
\end{lem}
\begin{proof}
A neighbour of $v_x$ can be uniquely represented by a lattice $L$ with $\pi (\fo_x)^2 \subsetneq L \subsetneq (\fo_x)^2$, hence corresponds uniquely by reduction modulo $\pi_x$ to a line in $\kappa(x)^2$, in fact a point in $\bP^1(\kappa(x))$.

The bijection from (b) to (a) is obtained by mapping $M$ to $L = M((\fo_x)^2)$ which is well defined since $M$ as an endomorphism of $(\fo_x)^2$ has $\coker(M)$ with $\fo_x$-length equal to $\ord_x(\det(M))$.
\end{proof}

By definition $G(R)$ is an $S$-arithmetic lattice for the set of places $S = \{0,1,\tau,\infty\}$:
\[
G(R) \inj G(K_{0}) \times G(K_{1}) \times G(K_{\tau}) \times G(K_\infty).
\]
Hence $G(R)$ is discrete and cocompact  by Behr and Harder, see \cite{margulis:book} I Theorem 3.2.4, because 
$G = \PGL_{1,D}$  is semisimple and anisotropic over $K$.
Since $D$ ramifies at $t = 0,1$, the local factors 
$G(K_0) \text{ and } G(K_{1})$ are compact, and 
\begin{equation} \label{eq:arithmeticlattice2factors}
G(R) \inj G(K_{\tau}) \times G(K_\infty)
\end{equation}
is still discrete and cocompact.

At $t=\tau$ and $t=\infty$ the quaternion algebra splits.
The field extension $\bF_q(z)/K$ used in Lemma~\ref{lem:split_z} maps the place $z= 0$ to $t = \infty$, so 
\[
K_{\infty} \simeq \bF_q((z))
\]
and the splitting of Lemma~\ref{lem:split_z} yields an isomorphism 
\[
G(K_\infty) \simeq  \PGL_2\big(\bF_q((z))\big)
\]
Similarly, the field extension $\bF_q(y)/K$ used in Lemma~\ref{lem:split_y} maps the place $y = 0$ to $t = \tau$, so 
\[
K_{\tau} \simeq \bF_q((y))
\]
and the splitting of Lemma~\ref{lem:split_y} yields an isomorphism 
\[
G(K_{\tau}) \simeq \PGL_2\big(\bF_q((y))\big) 
\]

The Bruhat Tits building of the two split factors $G(K_{\tau})$ and $G(K_\infty)$ is a tree $T_{\tau}$ resp.\ $T_\infty$ with constant valency $q+1$. 
It follows that the  diagonal embedding \eqref{eq:arithmeticlattice2factors} defines a simplicial action
\[
\rho :  G(R) \to \Aut(T_{\tau}) \times \Aut(T_{\infty}).
\]
that we call the Bruhat Tits action of $G(R)$ or its subgroups $\Gamma$, $\Lambda'$ and $\Lambda$.

\subsection{Vertex stabilizers}
For a group $H$ acting on a square complex such as a product of trees we denote the stabilizer of the vertex $v$ by $H_v$. 

\begin{prop} \label{prop:stab_v}
The following holds for the standard vertex $v_0 = (v_{\tau},v_\infty) \in T_{\tau} \times T_\infty$.
\begin{enumera}
\item
The stabilizer $G(R)_{v_0}$
agrees with the dihedral group $\Delta  = \langle d,s \rangle \subset G(R)$ of Lemma~\ref{lem:dihedralgroup}.
\item The stabilizer $G(R)_{v_0}$ acts on the set of neighbours of $v_{\tau} \in T_{\tau}$ (resp.\ of $v_\infty \in T_\infty$) 
as the normalizer of a non-split torus in $\PGL_2(\bF_q)$ with its natural action on $\bP^1(\bF_q)$ under the identification of 
Lemma~\ref{lem:BTneighbours}.
\item \label{propitem:transitiveonneighbours} The action of $\langle d \rangle$ is free and transitive on the set of neighbours of $v_{\tau}$ (resp.\ of $v_\infty$).
\end{enumera}
\end{prop}
\begin{proof}
For ease of legibility we write $v = v_0$. We must determine representatives for $d$ and $s$ in $\rM_2(\fo_x)$ for $x$ the places $t = c-1$ and $t = \infty$. 
Under both $\rho_y$ and $\rho_z$ we have 
\begin{equation} \label{eq:actionofZ}
Z \mapsto \matzz{0}{c}{1}{0}
\end{equation}
so that $\bF_q[Z]^\ast$ maps isomorphically onto a non-split torus in the constant subgroup
\[
\GL_2(\bF_q) \subset \GL_2(K_{\tau})  \quad \text{ (resp.\ } \GL_2(K_{\infty})  \text{ ) }.
\]
In view of Lemma~\ref{lem:BTneighbours} this shows  $d \in G(R)_v$, and also that (2) and (3) are consequences of (1). 

\smallskip

Concerning $s$, or more precisely its lift $F$, the formula in Lemma~\ref{lem:split_z} shows
\[
z \cdot \rho_z(F)  =   z \cdot \matzz{t(z-1)}{0}{0}{t(1-z)} =  \matzz{-1/2}{0}{0}{1/2} + \dO(z) \in \GL_2\big(\bF_q[[z]]\big)
\]
and the formula of Lemma~\ref{lem:split_y} shows
\[
\rho_y(F) = t(y-1) \cdot \matzz{\zeta_0}{-c\zeta_1}{\zeta_1}{-\zeta_0} = \tau \cdot \matzz{-\zeta_0}{c\zeta_1}{-\zeta_1}{\zeta_0}
+ \dO(y) \in \GL_2\big(\bF_q[[y]]\big).
\]
This proves $s \in G(R)_v$ and so $\Delta = \langle d,s \rangle \subseteq G(R)_v$.

\smallskip

For the converse inclusion $G(R)_v \subseteq \Delta$ we argue as follows. Since 
\[
(\rho_{y},\rho_z)(G(R)) \subset  \PGL_2\big(\bF_q((y))\big) \times \PGL_2\big(\bF_q((z))\big)
\]
is a discrete subgroup, and $\PGL_2(\bF_q[[y]]) \times \PGL_2(\bF_q[[z]])$ is compact, it follows that the stabilizer 
\[
G(R)_v = (\rho_{y},\rho_z)^{-1}\big(\PGL_2(\bF_q[[y]]) \times \PGL_2(\bF_q[[z]]) \big)
\]
is finite. The explicit splitting of Lemma~\ref{lem:split_z} shows that $G(R)_v$ is a finite subgroup of $\PGL_2(k)$ for the field $k=\bF_q(z)$. Moreover, since it is contained in a stabilizer of a vertex we also have an inclusion 
\[
G(R)_v \subseteq \PGL_2\big(\bF_q[[z]]\big).
\]
The kernel of $\PGL_2\big(\bF_q[[z]]\big) \surj \PGL_2(\bF_q)$, the evaluation in $z = 0$, 
is a pro-$p$-group so that 
\[
N = G(R)_v \cap  \ker\Big(\PGL_2\big(\bF_q[[z]]\big) \surj \PGL_2(\bF_q)\Big) 
\] 
is a normal $p$-subgroup of $G(R)_v$.

\smallskip

We first assume that  $N \not= \one$, so that  $p \mid \#G(R)_v$. Resorting to Dickson's classification of finite subgroups
of $\PGL_2(k)$ for a field $k$ (Dickson \cite{dickson:pgl2} actually classifies subgroups of $\PSL_2(\bF_q)$), see for example  \cite{faber:groupsinpgl2} Theorem~B, the presence of the (non-cyclic, but prime to $p$) dihedral group $\Delta$ in $G(R)_v$ shows that $G(R)_v$ is neither $p$-semi-elementary (in the sense of \cite{faber:groupsinpgl2}) nor isomorphic to $A_5$ (the only dihedral subgroups of $A_5$ are of order $2$, $4$, $6$ and $10$, so not of order $2(q+1)$). It remains the cases of $G(R)_v$ equal to 
\[
\PSL_2(\bF), \quad \PGL_2(\bF)
\]
for a subfield $\bF \subseteq \bF_q$. These groups have no non-trivial normal $p$-subgroup. In particular, in all cases $N = \one$. We thus search a group (by projecting to the constant part)
\[
\Delta \subseteq G(R)_v \subseteq \PGL_2(\bF_q).
\]
By \cite{giudici} Theorem 3.5 (\cite{giudici} treats $q > 3$; for $q=3$ the group $\Delta$ has order $8$ inside 
$\PGL_2(\bF_3)$ of order $24$, and so is maximal) $\Delta$ is a maximal subgroup, so that it remains only to contradict $G(R)_v = \PGL_2(\bF_q)$. 

\smallskip

For a group $H$ we denote by $H^\ab/2$ the maximal abelian quotient of exponent $2$. We argue by contradiction and assume $G(R)_v = \PGL_2(\bF_q)$. The reduced norm induces a diagram 
\[
\xymatrix@M+1ex{
\bF_q[Z]^\ast/\bF_q^\ast \ar[d] \ar[r] & \Delta \ar[r] \ar[d] & \PGL_2(\bF_q) \ar[d] \ar[r] & G(R) \ar[d]^{\Nrd} \\
\bF_q^\ast/(\bF_q^\ast)^2 \ar[r] & \Delta^\ab/2 \ar[r] & \PGL_2(\bF_q)^\ab/2 \ar[r] & R^\ast/(R^\ast)^2.
}
\]
Here $\Delta^\ab/2$ has dimension $2$ as a vector space over $\bF_2$ generated by the classes of $d$ and $s$ that are represented by $\mgen$ and $F$. 
Since $\Nrd(\mgen)$ is a generator of $\bF_q^\ast$ and $\Nrd(F) = t- t^2$ has odd valuation at $t= 0,1$, the induced map 
\[
\Delta^\ab/2 \inj R^\ast/(R^\ast)^2
\]
is injective. However, the group $\PGL_2(\bF_q)^\ab/2$ is cyclic of order $2$, a contradiction.
\end{proof}

\subsection{The lattice is arithmetic} We next show that $\Gamma$ is in fact of finite index in $G(R)$ and thus an arithmetic lattice.

\begin{prop} \label{prop:neighbours}
By means of the Bruhat Tits action, the set $A \subset \Gamma$ (resp.\ $B \subset \Gamma$) maps the distinguished vertex $v_0 = (v_{\tau},v_\infty) \in T \times T$ to the set of vertical (resp.\ horizontal) neighbours:
\begin{eqnarray*}
A.v_0 & = &   \{(v_{\tau},w) \ ; \ \text{ with } w \in T_{\infty} \text{ a neighbour of } v_{\infty} \}, \\
B.v_0 & = &  \{(w,v_\infty) \ ; \ \text{ with } w \in T_{\tau} \text{ a neighbour of } v_{\tau}\}. 
\end{eqnarray*}

\end{prop}
\begin{proof}
In view of Lemma~\ref{lem:BTneighbours} we must determine good representatives for $g \in A \cup B$ in $\rM_2(\fo_x)$ 
for $x$ the places $t = \tau$ and $t = \infty$. We set $\tilde{\alpha} = \alpha/t$ and $\tilde{\beta} = \zeta \beta/t$, and choose representatives
\begin{eqnarray*}
\tilde{\alpha}_\xi & = &  \xi \tilde{\alpha}Z \xi^{-1} = t^{-1} \alpha_{Z \wp(\xi)}, \\
\tilde{\beta}_\eta & = &  \eta \tilde{\beta}Z \eta^{-1} = t^{-1} \zeta \beta_{Z\wp(\eta)/\zeta}.
\end{eqnarray*} 
Here we used the identity  \eqref{eq:actiondAB} and the fact that $Z \wp(\xi)$ ranges over $N_c$ (resp.\ $Z\wp(\eta)/\zeta$ ranges over $M_\tau$). 
Because $\bF_q[Z]^\ast/\bF_q^\ast$ maps by Proposition~\ref{prop:stab_v} under both $\rho_y$ and $\rho_z$  isomorphically onto a non-split torus in the constant subgroup $\GL_2(\bF_q)$, it suffices to analyse $\tilde{\alpha}$ and $\tilde{\beta}$ 
in order for $\tilde{\alpha}_\xi$ and $\tilde{\beta}_\eta$ to be good representatives. Now by the formula in 
Lemma~\ref{lem:split_z}  
\begin{eqnarray*}
\rho_z(\tilde{\alpha}) & = & \matzz{0}{c(2-z)}{z}{0}, \\
\rho_z(\tilde{\beta}) & = &  \matzz{c\zeta_1}{c( \zeta_0 + 1-z)}{ \zeta_0  +z - 1}{c\zeta_1}, 
\end{eqnarray*}
and in Lemma~\ref{lem:split_y}
\begin{eqnarray*}
\rho_y(\tilde{\alpha}) & = & \matzz{c\zeta_1(y-1)}{c(1 + \zeta_0(1 - y))}{1+ \zeta_0(y-1)}{c\zeta_1(1-y)}, \\
\rho_y(\tilde{\beta}) & = &  \matzz{c\zeta_1 y}{c \zeta_0(2-y)}{\zeta_0 y}{c \zeta_1 (2-y)}.
\end{eqnarray*} 
This confirms the necessary integrality condition:
\[
\rho_z(\tilde{\alpha}_\xi), \rho_z(\tilde{\beta}_\eta) \in \rM_2(\bF_q[[z]]),
\]
\[
\rho_y(\tilde{\alpha}_\xi), \rho_y(\tilde{\beta}_\eta) \in \rM_2(\bF_q[[y]]).
\]
Since the reduced Norm transforms into the determinant when the quaternion algebra is split, we compute the valuation $\ord_x$ for $x$ the places  $t = \tau$ and $t = \infty$ of the determinants of these chosen integral representatives by means of \eqref{eq:normab} through
\[
\Nrd(\tilde{\alpha}_\xi) =   \frac{c^2}{t},
\]
\[
\Nrd(\tilde{\beta}_\eta) =   \frac{c^2}{\tau} \cdot  \frac{\tau - t}{t}.
\]
Indeed, we find 
\[
A.v_{\tau} = v_{\tau} \text{ and } A.v_{\infty} \subseteq \{\text{neighbours of } v_\infty\},
\]
\[
B.v_{\tau} \subseteq \{\text{neighbours of } v_{\tau}\} \text{ and } B.v_{\infty} = v_\infty.
\]
It remains to show that $A.v_\infty$ and $B.v_{\tau}$ contain  all the respective neighbours. 

\smallskip

For this we exploit the action by conjugation with $\langle d \rangle \subset G(R)_{v_0}$, see Proposition~\ref{prop:stab_v},  under which $A$ and $B$ are orbits, see Proposition~\ref{prop:sactsAB}.  For $g \in A \cup B$ it follows that 
\[
d(g)d^{-1}.v_0 = d.(g.v_0)
\]
so that we conclude by Proposition~\ref{prop:stab_v}~\ref{propitem:transitiveonneighbours}. 
\end{proof}

\begin{cor} \label{cor:transitiveaction}
\begin{enumera}
\item \label{coritem:transitive}
The Bruhat Tits action of $\Gamma$ on $T_\tau \times T_\infty$ is transitive on vertices.
\item  \label{coritem:Lambda}
We have $G(R) = \Lambda$.
\item \label{coritem:arithmeticlattice}
The group $\Gamma$ is a normal subgroup of finite index in $G(R)$ and thus an arithmetic lattice.
\end{enumera}
\end{cor}
\begin{proof}
Assertion \ref{coritem:transitive} follows from Proposition~\ref{prop:neighbours} together with the abstract criterion of 
Proposition~\ref{prop:mainabstractBMgroupresult}~\ref{propitem:abstracttransitiveaction}.  
Since the subgroup $\Gamma \subset G(R)$ acts transitively on $T_\tau \times T_\infty$, by  Proposition~\ref{prop:stab_v}  we have
\begin{equation} \label{eq:gammatimesstabv}
G(R) = \Gamma \cdot G(R)_v = \Gamma \cdot \Delta = \Lambda
\end{equation}
which is assertion \ref{coritem:Lambda}.  Proposition~\ref{prop:gammaoffiniteindexinlambda} and \ref{coritem:Lambda}  then show \ref{coritem:arithmeticlattice}.
\end{proof}

\subsection{Presentations of the arithmetic lattices}  \label{sec:determinepresentations}
Using VH-structures in groups we are now prepared to establish presentations for the arithmetic lattices of interest. That the groups $G(R) = \Lambda$, $\Lambda'$ and $\Gamma$ are finitely presented follows for example from \cite{behr} \S3.1. Indeed, the global rank of $G= \PGL_{1,D}$ is $0$ since our quaternion algebra $D$ is indeed ramified. Here we give explicit finite presentations of these groups.

\begin{thm} \label{thm:presentationgamma}
Let $1 \not= \tau \in \bF_q^\ast$, let $c \in \bF_q^\ast$ be a non-square and let  $\bF_q[Z]/\bF_q$ be the quadratic field extension  with $Z^2 = c$. The group $\Gamma_\tau$ of Definition~\ref{defi:lattice} is a torsion free arithmetic lattice in $G(K)$ with finite presentation 
\[
\Gamma_\tau = \left\langle 
\begin{array}{c}
{a}_\xi, {b}_\eta \text{ for } \xi,\eta \in \bF_q[Z]  \text{ with } \\[1ex]
N(\xi) = -c \text{ and } N(\eta) = \frac{c\tau}{1-\tau}  
\end{array}
\left|
\begin{array}{c}
a_\xi a_{-\xi} = 1, \ b_\eta b_{-\eta} = 1 \\[1ex]
\text{ and }  a_\xi b_\eta = b_\lambda a_\mu   \text{ if in } \bF_q[Z]: \\[1ex]
\eta = \lambda^q(\lambda + \xi)^{1-q}  \text{ and } \mu = \xi^q (\xi+\lambda)^{1-q}
\end{array} 
\right.
\right\rangle
\]
which acts  simply transitively via the Bruhat Tits action on the vertices of $T_{q+1} \times T_{q+1}$.
\end{thm}
\begin{proof}
This follows from Proposition~\ref{prop:neighbours} together with the abstract criterion of 
Proposition~\ref{prop:mainabstractBMgroupresult}~\ref{propitem:simplytransitive} and \ref{propitem:presentation}. The explicit shape of the relations of length $4$ was established in Corollary~\ref{cor:localpermutationwork}.
\end{proof}

Recall that $\mgen$ is a fixed chosen generator of the multiplicative group $\bF_q[Z]^\ast$,  and $d$ is its image in $\bF_q[Z]^\ast/\bF_q^\ast \subset D^\ast/K^\ast$.

\begin{thm} \label{thm:presentationLambda}
The arithmetic lattices  $G(R)$ and $G(R')$ have finite presentations as follows: with a generator $\mgen \in \bF_q[Z]^\ast$ and $\zeta  \in \bF_q[Z]^\ast$ such that $N(\zeta) = \frac{\tau-1}{\tau}$ we have 
\begin{eqnarray}
G(R)  & =  & \left\langle d,s,a,b \left|  
\begin{array}{c}
d^{q+1} = s^2 = (ds)^2 = 1, \ a^2 = b^2 = (s a)^2 = (s \zeta b)^2 = 1 \\[1ex]
(d^i a d^{-i})(d^{j} b d^{-j}) = (d^l b d^{-l})(d^{k} a d^{-k})  \\[1ex]
\text{ if } 0 \leq i,j,k,l \leq q \text{ satisfy } (\star) 
\end{array} 
\right. 
\right\rangle, \\[2ex]
G(R')  & = & \left\langle \  d,a,b  \ \left| 
\begin{array}{c}
d^{q+1} =  a^2 = b^2 = 1 \\[1ex]
(d^i a d^{-i})(d^{j} b d^{-j}) = (d^l b d^{-l})(d^{k} a d^{-k})  \\[1ex]
\text{ if } 0 \leq i,j,k,l \leq q \text{ satisfy } (\star) 
\end{array} 
\right.
\right\rangle,
\end{eqnarray}
where $(\star)$ is the condition in  $\bF_q[Z]^\ast/\bF_q^\ast$:
\[
(\star) = \left\{ \mgen^{j} = (\mgen^l -  \zeta \wp(\mgen^i) \cdot {\mgen^{ql}})Z \ \text{ and } \
\mgen^{k}  =  (\mgen^i -  \zeta^{-1} \wp(\mgen^l) \cdot \mgen^{qi})Z \right\}.
\]
In particular, $G(R) = \Lambda_\tau = \Gamma_\tau \rtimes \Delta$ and $G(R') = \Lambda'_\tau = \Gamma_\tau \rtimes \langle d \rangle$.
\end{thm}
\begin{proof}
By \eqref{eq:gammatimesstabv} of Corollary~\ref{cor:transitiveaction} we know that 
\[
G(R) = \Gamma \cdot \Delta = \Lambda.
\]
Since $\Gamma$ is torsion free, we have $\Gamma \cap \Delta = \one$, and since further $\Gamma$ is normal in $\Lambda$ by 
Corollary~\ref{cor:transitiveaction}~\ref{coritem:arithmeticlattice}, we conclude that $G(R) = \Lambda$ is the semi direct product of $\Delta$ acting on $\Gamma$ as claimed.

The semi-direct product structure for $\Lambda' = \Gamma \rtimes \langle d \rangle$ follows similarly. To show that $\Lambda' \subseteq G(R')$ is in fact an equality we estimate the index as $(\Lambda:\Lambda') \leq \#\langle s \rangle = 2$. But $s$ lifts to $F$ and 
\[
\Nrd(F) = t(t-1) \in K^\ast/(K^\ast)^2
\]
is nontrivial at $t = 1$ so that $s \notin G(R')$. Therefore 
\[
(G(R):G(R'))  = 2
\]
and $\Lambda' = G(R')$. 

The relations mixing $d,s$ and $a,b$  follow from the description of how the dihedral group $\Delta$ acts on $\Gamma$ given in Proposition~\ref{prop:sactsAB}, for example computing in $G(K)$:
\[
(sa)^2 = sa_Z Zsa = a_{-\bar Z} (sZs) a = a_Z Z a = a^2 = 1,
\]
\[
(s\zeta b)^2 = s \zeta b_{Z/\zeta}Z s \zeta b = s(\zeta b_{Z/\zeta} \zeta^{-1} )s Z b = sb_{Z/\bar \zeta} sZb = b_{Z/\zeta} Z b = b^2 = 1.
\]
The remaining relations follow by translating Theorem~\ref{thm:presentationgamma}: since $Z^2 = c$ is central we can lift 
\[
(d^i a d^{-i})(d^{j} b d^{-j})
\]
to $D^\ast$  as, using \eqref{eq:actiondAB} and $\wp(Z) = -1$, 
\[
(\mgen^i a_Z Z \mgen^{-i}) (\mgen^j b_{Z/\zeta} Z \mgen^{-j})/Z^2 = a_{Z\wp(\mgen^i)} Z b_{Z\wp(\mgen^j)/\zeta} Z^{-1} = a_{Z\wp(\mgen^i)} \cdot b_{-Z\wp(\mgen^j)/\zeta}.
\]
Therefore the four term relation 
\[
(d^i a d^{-i})(d^{j} b d^{-j}) = (d^l b d^{-l})(d^{k} a d^{-k})
\]
holds if and only if 
\[
a_{Z\wp(\mgen^i)} \cdot b_{-Z\wp(\mgen^j)/\zeta} = b_{Z\wp(\mgen^l)/\zeta} \cdot a_{-Z\wp(\mgen^k)},
\]
by  Corollary~\ref{cor:localpermutationwork} if and only if 
\begin{eqnarray*}
- Z\wp(\mgen^j)/\zeta & = & \ov{Z\wp(\mgen^l)/\zeta} \cdot \big(Z\wp(\mgen^l)/\zeta -  Z\wp(\mgen^i)\big)^{1-q}, \\
- Z\wp(\mgen^k) & = & \ov{Z\wp(\mgen^i)} \cdot \big(Z\wp(\mgen^i) -  Z\wp(\mgen^l)/\zeta\big)^{1-q}.
\end{eqnarray*}
These equations are equivalent to 
\begin{eqnarray*}
\wp(\mgen^{j+l}/\zeta) & = & \wp\big(Z\wp(\mgen^l)/\zeta -  Z\wp(\mgen^i)\big), \\
\wp(\mgen^{k+i}) & = & \wp\big(Z\wp(\mgen^i) -  Z\wp(\mgen^l)/\zeta\big),
\end{eqnarray*}
and further to condition $(\star)$ in $\bF_q[Z]^\ast/\bF_q^\ast$ due to Lemma~\ref{lem:hilb90} and $\delta^{1+q} \in \bF_q^\ast$.
\end{proof}

\section{Group theory of the quaternionic lattice} \label{sec:finitegrouptheory}
In this section we collect various group theoretic information on our arithmetic lattices.

\subsection{Local structure} \label{sec:localstructure}
The VH-structure $A,B$ in $\Gamma$ determines two permutation groups as follows. There are unique maps of sets 
\[
A  \to \Aut(B), \qquad a  \mapsto \sigma^B_a, 
\]
\[
B \to \Aut(A), \qquad b \mapsto  \sigma^A_b.
\]
such that in $\Gamma$:
\[
a \sigma^B_a(b) = b \sigma^A_b(a).
\]
In the corresponding square complex we find squares 
\[
\xymatrix@M0ex@R-2ex@C-2ex{
\bullet \ar[r]^(1){\sigma_a^B(b)} & \ar@{-}[r] & \bullet \\
\ar@{-}[u] & & \ar@{-}[u] \\
\bullet \ar[r]_(1){b} \ar[u]^(1){a} & \ar@{-}[r] & \bullet \ar[u]_(1){\sigma_b^A(a)} \\
}
\]
The maps $a \mapsto \sigma_a^B$ and $b \mapsto \sigma_b^A$ are called the \textbf{local permutation structure} of the VH-structure $A,B \subset \Gamma$, and the \textbf{local structure} is given by the permutation groups generated by the images
\begin{eqnarray*}
P_A & = & \langle \sigma^A_b \ ; \ \text{ all } b \in B \rangle \\
P_B & = & \langle \sigma^B_a \ ; \ \text{ all } a \in A \rangle
\end{eqnarray*}
together with the permutation action on $B$ respectively on $A$. These two permutation groups are nothing but the 
local groups  in the sense of \cite{burger-mozes:lattices}~\S6.1 associated to  $\Gamma$ as a lattice in 
$\Aut(T_{q+1}) \times \Aut(T_{q+1})$.

\subsubsection{The local permutations}  
The algebraic origin of our lattice allows us to determine these permutation groups. Recall that the sets $A,B$ of the VH-structure of $\Gamma$ are parametrized by the set of rational points of a variety isomorphic to $\bP^1$:
\begin{eqnarray*}
\bT_{-c}(\bF_q)  & = & A, \\
\bT_{\frac{c\tau}{1-\tau}}(\bF_q)  & = & B
\end{eqnarray*}
via $\xi \mapsto a_\xi$ and $\eta \mapsto b_\eta$, see Definition~\ref{defi:lattice}.

\begin{prop} \label{prop:localperm}
\begin{enumera}
\item \label{propitem:permutationstructureofGammaquotient}
The local permutation structure of the VH-structure $A,B \subset \Gamma$ is given by  the maps $\xi \mapsto \sigma_\xi^B$ and $\eta \mapsto \sigma_\eta^A$ defined by the following elements of $\Aut(\bT_{\frac{c\tau}{1-\tau}})$ respectively $\Aut(\bT_{-c})$:
\begin{eqnarray*}
\sigma^A_\eta:  w & \mapsto & \sigma_{\eta}(w) = \bar{w} \cdot \wp(w - \eta), \\ 
\sigma^B_\xi: w & \mapsto & \sigma_\xi(w) = \bar{w} \cdot \wp(w - \xi),
\end{eqnarray*}
\item For $\lambda \in \bT(\bF_q)$ we have the following conjugation relations
\begin{eqnarray*}
\sigma_{\lambda \xi}^B & = & \lambda \sigma_\xi^B \lambda^{-1}, \\
\sigma_{\lambda \eta}^A & = & \lambda \sigma_\eta^A \lambda^{-1},
\end{eqnarray*}
where $\lambda$ and $\lambda^{-1}$ are the multiplication maps by $\lambda$ and $\lambda^{-1}$ respectively.
\item
As elements of $\bF_q^\ast/(\bF_q^\ast)^2$ we have determinants, for $\xi \in \bT_{-c}(\bF_q)$ and $\eta \in \bT_{\frac{c\tau}{1-\tau}}$
\begin{eqnarray*}
\det(\sigma_\eta^A) & = & \frac{1}{1-\tau},\\
\det(\sigma_\xi^B) & = & \frac{1}{\tau}.
\end{eqnarray*}
\end{enumera}
\end{prop}
\begin{proof}
Assertion (1) recalls Corollary~\ref{cor:localpermutationwork}, and assertion (2) is an  elementary computation
in $\GL(\bF_q[Z])$. 

By assertion (2) all local permutations for $\xi \in A$ (resp.\ $\eta \in B$) are conjugate and thus share the same determinant. It thus suffices to compute one of them by Lemma~\ref{lem:detsigma_xi} as 
\[
\det(\sigma_\eta^A)  =  \det(\sigma_{Z/\zeta})  =  \det( w \mapsto w - \frac{Z/\zeta}{Z} \bar w) = 1 - N(1/\zeta) = 1 - \frac{\tau}{\tau-1} = \frac{1}{1-\tau},
\]
\[
\det(\sigma_\xi^B)  =  \det(\sigma_{Z})  =  \det( w \mapsto w - \frac{Z}{Z/\zeta} \bar w) = 1 - N(\zeta)  = 1 - \frac{\tau-1}{\tau} = \frac{1}{\tau} ,
\]
and assertion (3) follows.
\end{proof}

\begin{cor} \label{cor:explicitclassifyingspace}
The classifying space of $\Gamma_\tau$  has a realization as the finite square complex
\[
S_{\Gamma_\tau} = \Gamma_\tau \backslash T_{q+1} \times T_{q+1}
\]
with the following explicit description.
\begin{enumera}
\item 
Vertices: There is only one vertex.
\item 
Edges: There is a vertical oriented edge $a_\xi$ for every $\xi \in \bF_q[Z]$ of norm $N(\xi) = -c$, a horizontal oriented edge $b_\eta$ for every $\lambda \in \bF_q[Z]$ of norm $N(\lambda) = \frac{c\tau}{1-\tau}$. The orientation reversion map is $a_\xi \mapsto a_{-\xi}$ and $b_\lambda \mapsto b_{-\lambda}$.
\item 
Squares: For every pair $a_\xi$, $b_\lambda$ of an oriented horizontal and an oriented vertical edge there is a square in $S_{\Gamma_\tau}$
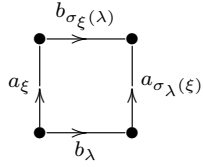
\[
\xymatrix@M0ex@R-2ex@C-2ex{
\bullet \ar[r]^(1){b_{\sigma_\xi(\lambda)}} & \ar@{-}[r] & \bullet \\
\ar@{-}[u] & & \ar@{-}[u] \\
\bullet \ar[r]_(1){b_\lambda} \ar[u]^(1){a_\xi} & \ar@{-}[r] & \bullet \ar[u]_(1){a_{\sigma_\lambda(\xi)}} \\
}
\]
each of which is constructed four times, but in fact occurs in $S_{\Gamma_\tau}$ only once.
\end{enumera}
\end{cor}
\begin{proof}
This follows immediately form Proposition~\ref{prop:localperm} \ref{propitem:permutationstructureofGammaquotient} and the fact recalled in \cite{burger-mozes:lattices}~\S6.1 page 181 that the local permutation structure encodes the underlying square complex with one vertex. 
\end{proof}

\subsubsection{The local groups} So far our lattices $\Gamma_\tau$ have been quite uniform in the parameter $\tau$. The local groups attached to the combinatorial structure of a VH-structure in a group shows some mild diophantine dependence on $\tau$.

\begin{prop} \label{prop:determinelocalstructure}
The local structure of the VH-structure $A,B \subset \Gamma_\tau$ is as follows:
\begin{eqnarray*}
P_A & = & \left\{\begin{array}{c}
\PGL(\bF_q[Z]) \\
\PSL(\bF_q[Z]) 
\end{array}
\right.
 \text{ if } 1-\tau \text{ is }
\left\{\begin{array}{c}
\text{ not a square }  \\
\text{ a square } 
\end{array}
\right.
\text{ in } \bF_q^\ast \\ 
P_B & = & \left\{\begin{array}{c}
\PGL(\bF_q[Z]) \\
\PSL(\bF_q[Z]) 
\end{array}
\right.
 \text{ if } \tau \text{ is }
\left\{\begin{array}{c}
\text{ not a square }  \\
\text{ a square } 
\end{array}
\right.
\text{ in } \bF_q^\ast
\end{eqnarray*}
as permutation groups acting naturally on $\bP(\bF_q[Z])$ when identified with $A$ (resp.\ $B$).
\end{prop}
\begin{proof}
Let us abbreviate $H = P_A$ or $P_B$. By Proposition~\ref{prop:localperm} this is a subgroup 
\[
H \subseteq \PGL(\bF_q[Z])
\]
is generated by the orbit  of a non-trivial element under conjugation by the non-split torus
\[
C := \bF_q[Z]^\ast/\bF_q^\ast \subseteq \Aut(\bT_u) \simeq \PGL(\bF_q[Z])
\]
for $u=-c$ (resp.\ $u = \frac{c\tau}{1-\tau}$).
Let $N \subseteq \PGL(\bF_q[Z])$ be the normalizer of $H$. It follows that $C  \subseteq N$. By inspection of the list of maximal subgroups \cite{giudici} Theorem 3.5, it follows that $N$ is either contained in the normalizer of $C$ or $N = \PGL(\bF_q[Z])$.

\smallskip

If $N$ is contained in the normalizer of $C$, a dihedral group, then in $N$ conjugacy classes have size $\leq 2$. On the other hand, the natural generating set of $H$ has size $q+1$ and  is contained in a conjugacy class of $N$, a contradiction.

\smallskip

We conclude that $H$ is a normal subgroup of $\PGL(\bF_q[Z])$. This leaves $H = \PGL(\bF_q[Z])$ or $H = \PSL(\bF_q[Z])$ because $\PSL(\bF_q[Z])$ is simple if $q > 3$ (and if $q=3$ by inspection of the normal subgroups of $\PGL_2(\bF_3) \simeq S_4$ with at least $4 = q+1$ non-zero elements).  

\smallskip

In order to distinguish between $\PGL$ and $\PSL$ we simply have to evaluate the homomorphism
\[
\ov{\det} : \PGL(\bF_q[Z])  \surj \bF_q^\ast/(\bF_q^\ast)^2
\]
induced by the determinant. This was done in Proposition~\ref{prop:localperm} which completes the proof.
\end{proof}

\begin{rmk}
(1)
Proposition~\ref{prop:determinelocalstructure} shows that the local permutation groups of the defining VH-structure in $\Gamma$ are $2$-transitive, hence quasi-primitive, see \cite{burger-mozes:lattices}~\S0.1.

(2) 
All four possibilities that Proposition~\ref{prop:determinelocalstructure} gives for the pair $(P_A,P_B)$ actually occur for suitable parameter $q$ and $\tau$.
\end{rmk}
\subsection{Some abelian quotients}  Here are two finite abelian quotients of $\Gamma$ that occur uniformly thorugh our series of examples.

\begin{prop} \label{prop:uniformabelianquot}
The group $\Gamma$ has the following finite abelian quotients.
\begin{enumera}
\item The assignment $a_\xi  \mapsto \xi$ and $b_\eta  \mapsto \eta$ defines a surjective map
\[
\Gamma \surj \bF_p[Z].
\]
\item \label{itemprop:homotoZmod2Zmod2}
The assignment $a_\xi  \mapsto  \genfrac{(}{)}{0pt}{}{1}{0}$ and $b_\eta  \mapsto  \genfrac{(}{)}{0pt}{}{0}{1}$ defines a surjective map
\[
\Gamma \surj \bF_2  \oplus \bF_2.
\]
\end{enumera}
\end{prop}
\begin{proof}
(1) 
We need to check the relations of $\Gamma$. First  we check
\[
a_\xi a_{-\xi} \mapsto  \xi - \xi =  0,
\]
\[
b_\eta b_{-\eta} \mapsto \eta - \eta = 0.
\]
Next, if $a_\xi b_\eta = b_\lambda a_\mu$, then $\xi,\eta,\lambda,\mu$ are linked in particular by the second equation of  \eqref{eq:comparecoefficientsABequalsBA}:
\[
\xi + \eta = \lambda +\mu,
\]
so the words $a_\xi b_\eta$ and $b_\lambda a_\mu$ are mapped to the same element. This finishes the proof of (1).

\smallskip

(2) Checking the relations for (2) is obvious. The resulting homomorphism counts the parity of the number of generators $a_\xi$ and the parity of the number of generators $b_\eta$ occurring in a word representing an element of $\Gamma$. This is clearly well defined.
\end{proof}

\begin{rmk}
(1)
By Proposition~\ref{prop:uniformabelianquot}, the maximal abelian quotient of $\Gamma$ has a quotient isomorphic to 
$\bZ/2p\bZ \times \bZ/2p\bZ$ and, at least experimentally, is at most by a factor of $2$ larger. The additional factor $2$, if it occurs at all, comes from either an extra factor or by doubling a factor of the above quotient.

(2) The homomorphism of Proposition~\ref{prop:uniformabelianquot}~\ref{itemprop:homotoZmod2Zmod2} can also be described as the reduced norm modulo squares 
\[
\Nrd : \Gamma \surj \langle t, t\frac{t-\tau}{1-\tau}  \rangle \subset R'^\ast/(R'^\ast)^2,
\]
see \eqref{eq:normab}, 
the kernel of which describes the maximal subgroup $\Gamma^1 \subset \Gamma$ that lifts to the universal cover $\SL_{1,D} \to \PGL_{1,D}$.
\end{rmk}

\subsection{Residually pro-$p$}
As an arithmetic lattice in characteristic $p$ the group $\Gamma$ is virtually residually pro-$p$. In fact, $\Gamma$ itself is residually pro-$p$ already.

\begin{prop} \label{prop:resprop}
The group $\Gamma$ is residually pro-$p$.
\end{prop}
\begin{proof}
We complete at $t = 1$ and consider the $\bF_q[[t-1]]$-order in the skew-field $D \otimes_K \bF_q((t-1))$
\[
\hat{\fO}'  = \fO' \otimes_{R'} \bF_q[[t-1]].
\]
Since $F^2 = t(t-1)$, the two-sided ideals $(F^n) \subseteq \hat{\fO}'$ form an exhausting $(t-1)$-adic filtration of 
$\hat{\fO}'$ with successive quotients isomorphic as $\hat{\fO}'$-module to $\hat{\fO}'/(F)$ which itself is isomorphic to
\[
\hat{\fO}'/(F) \xrightarrow{\sim} \bF_q[Z]
\]
via $Z \mapsto Z$ and $F \mapsto 0$. It follows that the kernel of the modulo $F$ reduction map 
\[
\ker \big(\hat{\fO}'^\ast  \surj \bF_q[Z]^\ast\big)
\]
is a pro-$p$ group. An easy snake lemma shows that also 
\[
\ker \big(\hat{\fO}'^\ast /\bF_q[[t-1]]^\ast \surj \bF_q[Z]^\ast/\bF_q^\ast\big)
\]
is a pro-$p$ group. We furthermore know that 
\[
\Gamma \subset \Lambda'   = G(R') \subseteq \hat{\fO}'^\ast/\bF_q[[t-1]]^\ast
\]
so that it remains to show the vanishing of  the composite
\[
\Gamma \to \Lambda' \to \hat{\fO}'^\ast/\bF_q[[t-1]]^\ast \to  \bF_q[Z]^\ast/\bF_q^\ast.
\] 
This follows form 
\begin{eqnarray*}
a_\xi = (tZ + \xi F)Z & \mapsto & Z^2 = c, \\
b_\eta = (tZ + \eta F)Z & \mapsto & Z^2 = c
\end{eqnarray*}
which completes the proof.
\end{proof}

\section{Classification --- consequences of rigidity} \label{sec:rigidity}
In this section we spell out the consequences of Margulis rigidity theory as in \cite{margulis:book} for the groups $\Gamma_\tau$ of Section~\S\ref{sec:lattice}. By Corollary~\ref{cor:transitiveaction}~\ref{coritem:arithmeticlattice} and 
Theorem~\ref{thm:presentationgamma} the $\Gamma_\tau$ are torsion free arithmetic lattices for the group $G=\PGL_{1,D}$ with an explicit finite presentation. As an immediate consequence of strong approximation (for $\SL_2$) the lattices $\Gamma_\tau$ are irreducible.

\subsection{Commensurability classification}
Recall that two groups are commensurable if they share isomorphic finite index subgroups.

\begin{prop} \label{prop:commensurabilityclassification}
The lattice $\Gamma_\tau$ is commensurable with $\Gamma_{\secondtau}$ if and only if $\secondtau$ is in the $\Gal(\bar \bF_p/\bF_p)$-orbit of $\tau$ or $1-\tau$, i.e., if there is an $k \in \bZ$ such that $\secondtau = \tau^{p^k}$ or $\secondtau = 1 - \tau^{p^k}$.
\end{prop}
\begin{proof}
We will decorate all notation with a hat when it refers to $\Gamma_{\secondtau}$. 
Let us start with a finite index subgroup $\Gamma_0 \subset \Gamma_\tau$ and a nontrivial homomorphism 
$\ph: \Gamma_0 \to \Gamma_{\secondtau}$. By \cite{margulis:book} Chapter VIII Theorem C the homomorphism gives rise to an embedding of fields 
\[
\sigma : K \inj \secondK
\]
and a homomorphism 
\[
f: \PGL_{1,D} \otimes_{K,\sigma} \secondK \to \PGL_{1,\secondD}
\]
such that the following commutes:
\[
\xymatrix@M+1ex@R-2ex{
\PGL_{1,D}(K) \ar[r]^{f(\secondK)} & \PGL_{1,\secondD}(\secondK) \\
\Gamma_0 \ar[r]^\ph \ar@{}[u]|{\displaystyle \cup} & \Gamma_{\secondtau}. \ar@{}[u]|{\displaystyle \cup}
}
\]
(Note that a central twist is necessarily trivial because $\PGL_{1,D}$ has trivial center.)  

Now we assume that $\ph$ is an isomorphism onto a finite index subgroup $\secondGamma_0 \subseteq \Gamma_{\secondtau}$, so that the inverse $\ph^{-1} : \secondGamma_0 \to \Gamma_\tau$ gives rise to the inverse embedding 
\[
\sigma^{-1} : \secondK \inj K
\]
Composing $\ph$ with $\ph^{-1}$ yields the identity, so that by uniqueness the maps $\sigma$ and $\sigma^{-1}$ must be inverses of each other. The map $\sigma$ induces an automorphism $\Frob^k$ on the field of constants $\bF_q \subset K,\secondK$, and $\sigma$ itself then encodes a semi-linear automorphism $\sigma : \bP^1_{\bF_q} \to \bP^1_{\bF_q}$, that maps 
\[
\secondS = \{0,1,\secondtau,\infty\}
\]
to $S = \{0,1,\tau,\infty\}$. Changing the $\bF_q$ structure on one copy of $\bP^1$ by a suitable power of Frobenius we obtain a linear automorphism of $\bP^1_{\bF_q}$  that maps $\secondS$ to $\{0,1,\tau^{p^k},\infty\}$ and, since furthermore the places $0,1$ encode the isomorphism type of $D$, hence $\PGL_{1,D}$,  must also preserve $\{0,1\}$ as a set. It follows that up to applying a suitable power of Frobenius, the commensurability class of $\Gamma_\tau$ determines a rational partitioned $4$-tuple 
\[
((r_1,r_2);(t_1,t_2))
\]
of distinct elements of $\bP^1(\bF_q)$
up to permutations that preserve the partition in $(r_1,r_2)$ and $(t_1,t_2)$. The moduli problem of such structured $4$-tuples is covered by $\cM_{0,4}$, the moduli space of genus $0$ curves with an ordered set of $4$ distinct points, and via the double ratio
\[
\lambda = \DV(r_1,r_2;t_1,t_2) = \frac{r_1-t_1}{r_2-t_1} : \frac{r_1-t_2}{r_2-t_2} 
\]
isomorphic to $\bP^1 - \{0,1,\infty\}$. The remaining freedom is a Klein-four-group generated by transpositions $r_1 \leftrightarrow r_2$ and $t_1 \leftrightarrow t_2$. This group acts on the double ratio $\lambda$ through the single involution 
\[
\lambda \mapsto 1/\lambda.
\]
Because of 
\[
\DV(0,1;\tau,\infty) = \frac{\tau}{\tau-1}
\]
the effect on $\tau$ is therefore 
\[
\tau \mapsto 1 - \tau.
\]
This shows that commensurable groups $\Gamma_\tau$ and $\Gamma_{\secondtau}$ have their parameters $\tau$ and 
$\secondtau$ linked as predicted by the proposition. The converse is clear by the discussion above.
\end{proof}

\subsection{Isomorphism classification} We establish two explicit isomorphisms of lattices thereby showing that for the 
$\Gamma_\tau$ the commensurability classification agrees with the a priori finer classification up to isomorphism.

\begin{prop} \label{prop:1minustau}
There is an isomorphism
\[
\Gamma_\tau \simeq \Gamma_{1-\tau}
\]
coming from an isomorphism of the defining presentations.
\end{prop}
\begin{proof}
The isomorphism will interchange vertical with horizontal edges. We thus have to find maps
\[
A_\tau = \bT_{-c}(\bF_q)  \xrightarrow{\ph} \bT_{\frac{c(1-\tau)}{\tau}}(\bF_q) = B_{1-\tau}
\]
\[
B_\tau = \bT_{\frac{c\tau}{1-\tau}}(\bF_q)  \xrightarrow{\psi} \bT_{-c}(\bF_q) = A_{\tau}
\]
such that $(\ph,\psi)$ preserves the local permutation structure, i.e., the following diagrams commute:
\[
\xymatrix@M+1ex{
A_\tau \ar[d]_{\sigma^{B_\tau}} \ar[r]^\ph & B_{1-\tau} \ar[d]^{\sigma^{A_{1-\tau}}} \\
\Sym(B_\tau) \ar[r]^{\psi(-)\psi^{-1}} & \Sym(A_{1-\tau})
}
\qquad
\xymatrix@M+1ex{
B_\tau \ar[d]_{\sigma^{A_\tau}} \ar[r]^\psi & A_{1-\tau} \ar[d]^{\sigma^{B_{1-\tau}}} \\
\Sym(A_\tau) \ar[r]^{\ph(-)\ph^{-1}} & \Sym(B_{1-\tau})
}
\]
It follows that the map of generators
\[
A_\tau \sqcup B_\tau \xrightarrow{\ph \sqcup \psi} B_{1-\tau} \sqcup A_{1-\tau}
\]
induces an isomorphism $\Gamma_\tau \simeq  \Gamma_{1-\tau}$.

\smallskip

In view of Proposition~\ref{prop:localperm} the maps $\ph$ and $\psi$ must satisfy for all $\xi \in A_\tau$ and $\eta \in B_\tau$
\begin{eqnarray} \label{eq:conditionisotau1minustau}
\psi(\bar{\eta} \cdot \wp(\eta - \xi)) & = & \ov{\psi(\eta)} \cdot \wp(\psi(\eta) - \ph(\xi)), \\
\ph(\bar{\xi} \cdot \wp(\xi - \eta)) & = & \ov{\ph(\xi)} \cdot \wp(\ph(\xi) - \psi(\eta)). \notag
\end{eqnarray}
Let $\zeta \in \bF_q[Z]^\ast$ be an element of norm 
\[
N(\zeta) = \frac{\tau-1}{\tau}
\]
which exists by Lemma~\ref{lem:hilb90}.
We make the ansatz
\[
\ph(\xi) = \zeta \cdot \xi \quad \text{ and } \quad \psi(\eta) = \zeta \cdot \eta
\]
which indeed defines well defined maps between the respective sets. The trivial verification of \eqref{eq:conditionisotau1minustau} for these choices of $\ph$ and $\psi$ based on the multiplicativity of $\wp$ completes the proof.
\end{proof}

\begin{prop} \label{prop:taupowerp}
There is an isomorphism
\[
\Gamma_\tau \simeq \Gamma_{\tau^p}
\]
coming from an isomorphism of the defining presentations.
\end{prop}
\begin{proof}

We argue as in Proposition~\ref{prop:1minustau} that there are maps
\[
A_\tau = \bT_{-c}(\bF_q)  \xrightarrow{\ph}  \bT_{-c}(\bF_q) =  A_{\tau^p}
\]
\[
B_\tau = \bT_{\frac{c\tau}{1-\tau}}(\bF_q)  \xrightarrow{\psi} \bT_{\frac{c\tau^p}{1-\tau^p}}(\bF_q) = B_{\tau^p}
\]
such that $(\ph,\psi)$ preserves the local permutation structure. Here this amounts for all $\xi \in A_\tau$ and $\eta \in B_\tau$ to 
\begin{eqnarray} \label{eq:conditionisotaupowerp}
\psi(\bar{\eta} \cdot \wp(\eta - \xi)) & = & \ov{\psi(\eta)} \cdot \wp(\psi(\eta) - \ph(\xi)), \\
\ph(\bar{\xi} \cdot \wp(\xi - \eta)) & = & \ov{\ph(\xi)} \cdot \wp(\ph(\xi) - \psi(\eta)). \notag
\end{eqnarray}
Define $\ph$ and $\psi$ as
\begin{eqnarray*}
\ph(\xi) & = & \frac{\xi^p}{c^{(p-1)/2}}, \\
\psi(\eta) & = & \frac{\eta^p}{c^{(p-1)/2}}
\end{eqnarray*}
which indeed defines well defined maps between the respective sets. Again, the trivial verification of 
\eqref{eq:conditionisotaupowerp} for these choices of $\ph$ and $\psi$ based on $\wp((-)^p) = \wp(-)^p$ completes the proof.
\end{proof}

\begin{cor} \label{cor:isomclass}
Lattices of the form $\Gamma_\tau$ are commensurable if and only if they are isomorphic. They are isomorphic if and only if suitable  Galois conjugates of the parameters $\tau$ agree or add up to $1$.
\end{cor}
\begin{proof}
This follows immediately from Proposition~\ref{prop:1minustau}, Proposition~\ref{prop:taupowerp} and 
Proposition~\ref{prop:commensurabilityclassification}.
\end{proof}

\subsection{Rank and finiteness properties}

The lattice $\Gamma_\tau$ is an $S$-arithmetic lattice for the adjoint group $G = \PGL_{1,D}$ over $K = \bF_q(t)$ and 
\[
S = \{0,1,\tau,\infty\},
\]
hence an $S$-arithmetic lattice in $\prod_{v \in S} G(K_v)$ of rank 
\[
\sum_{v \in S} \text{rank}_v (G) = 2.
\]
Indeed, the ramified places of $D$ have anisotropic $G(K_v)$ and the local factor at the other two places are isomorphic to  $\PGL_2\big(\bF_q((t))\big)$ and thus each contribute rank $1$. 

\smallskip

Recall that a group is called FAB if all finite index subgroups $H$ have finite maximal abelian quotient $H^\ab$.
 
\begin{prop} \label{prop:justinfinite}
The non-trivial normal subgroups of $\Gamma_\tau$ are of finite index. In particular the following holds.
\begin{enumera}
\item \label{propitem:finitemaxab}
The maximal abelian quotient $\Gamma_\tau^\ab$ is finite.
\item \label{propitem:fab}
The group $\Gamma_\tau$ is a FAB group.
\end{enumera}
\end{prop}
\begin{proof}
The lattice $\Gamma_\tau$ has rank $2$. Therefore, by \cite{margulis:book} Chapter VIII Theorem A, non-trivial normal subgroups are either of finite index or central, hence finite. The proposition follows because $\Gamma_\tau$ is torsion free.
Note that assertion~\ref{propitem:finitemaxab} of the proposition is a special case of  \cite{margulis:book} Proposition~IX.5.3, and assertion~\ref{propitem:fab} follows because the same reference applies also to all finite index subgroups of $\Gamma_\tau$.
\end{proof}

\section{Construction of a non-classical fake quadric} \label{sec:fakequadric}

\subsection{Fake quadrics}
The motivating question that leads to the notion of a \textbf{fake quadric} asks for compact complex K\"ahler manifolds with the same Betti numbers $\rb_1 = 0$ and $\rb_2 = 2$ as $\bP^1 \times \bP^1$ and with geometric genus $p_g = 0$. These surfaces are either Hirzebruch surfaces $\Sigma_n = \bP(\dO \oplus \dO(n)) \to \bP^1$ or algebraic surfaces of general type. In the latter case these surfaces are either the blow-up of a closed point on a fake $\bP^2$ or are called fake quadric if they are minimal. In characteristic $0$ the numerical conditions are equivalent to 
\[
\rc_1^2 = 8, \quad \rc_2 = 4, \quad \text{ and  trivial Albanese variety},
\]
and the assumption of being minimal of general type translates into the canonical bundle being ample.  In characteristic $0$ then also $\Pic^0$ and $\rH^1(\dO)$ vanish.

All known examples of fake quadrics $X$ are $\bC$-analytically uniformized by  $\bH \times \bH$, the product of two copies of the upper half plane $\bH$, so that $\pi_1(X) \subset \PSL_2(\bR) \times \PSL_2(\bR)$ is a cocompact lattice.
\begin{enumeral}
\item Reducible lattices: these are Beauville surfaces $X = (C_1 \times C_2)/G$ for suitable smooth projective curves $C_i$ of genus $\geq 2$ and a suitable finite group $G$. The first such surface was found by Beauville \cite{beauville:surfaces} exercise X.13 (4) with $G= (\bZ/5\bZ)^2$ acting on a product of Fermat quintic curves, and all such fake quadrics have been classified by Bauer, Catanese and Grunewald \cite{bauercatanesegrunewald}. 
\item  Irreducible lattices: here $\pi_1(X) \subset \PSL_2(\bR) \times \PSL_2(\bR)$ is an irreducible quaternionic lattice of suitable covolume, see Kuga and Shavel \cite{shavel:fakequadric}, and D\v{z}ambi\'c \cite{dzambic:fakequadrics}.
\end{enumeral}

It remains an open question very much in the spirit of the motivation above whether there is a complex  fake quadric homeomorphic to $\bC\bP^1 \times \bC\bP^1$ or at least with finite or trivial $\pi_1$,  see Hirzebruch \cite{hirzebruch:oevre1}  subproblem of Problem 25 on page 779.

\subsection{Non-classical fake quadrics} We will work in positive characteristic and deal with smooth projective surfaces with non-classical properties.

Following \cite{mumford:fakep2} we construct smooth projective surfaces in characteristic $p \geq 3$ by uniformization in formal geometry over $\Spf(\bF_q[[t]])$. For $q=3$ the construction yields a \textbf{non-classical fake quadric} $X$ in characteristic $3$, i.e., a smooth projective surface with ample canonical bundle, trivial Albanese variety and the same Chern numbers as $\bP^1 \times \bP^1$ 
\[
\rc_1^2 = 8, \quad \rc_2 = 4.
\]
The fake quadric will be non-classical because $\rH^1(X,\dO_X) \not= 0$ and so the Picard scheme $\Pic^0_X$ is non-reduced of dimension $0$.

\subsection{The wonderful scheme}
In this section we reset notation and choose a complete discrete valuation ring $R$ with uniformizer $\pi$, \textbf{finite} residue field $k = R/(\pi)$ with $q$ elements, and field of fractions $K = R[\frac{1}{\pi}]$.

For $L = (L_{ij})_{0 \leq i,j \leq 1} \in \GL_2(K)$ we define linear forms $\ell_i = \ell_i^L$ in the variables $x_0$, $x_1$  by matrix multiplication as 
\[
L\genfrac{(}{)}{0pt}{}{x_0}{x_1} = \genfrac{(}{)}{0pt}{}{\ell_0}{\ell_1}.
\]
The affine $R$-scheme defined by  the subring $R\left[\frac{\ell_0}{\ell_1}, \pi \frac{\ell_1}{\ell_0}\right] \subset K(\frac{x_0}{x_1})$ is 
\[
\tilde{Y}_L = \Spec(R\left[\frac{\ell_0}{\ell_1}, \pi \frac{\ell_1}{\ell_0}\right]) \simeq \Spec R[u,v]/(uv - \pi),
\]
hence is a regular scheme and semistable over $\Spec(R)$. 
The special fibre
$\Spec k[u,v]/(uv)$
is a union of affine lines  transversally glued in $(0,0)$. We set 
\[
Y_L \subset \tilde{Y}_L
\]
for the complement of all the finitely many $k$-rational points of the special fibre but keeping the singular point $(0,0)$. Then $Y_L$ actually only depends on the image of $L$ in $\PGL_2(K)$, and  we set 
\[
Y = \bigcup_{L \in \PGL_2(K)} Y_L
\]
as the separated $R$-scheme obtained by glueing the $Y_L$ respecting the function field $K(\frac{x_0}{x_1})$. The generic fibre $Y_K = Y \otimes_R K$ is via the coordinate $x_0/x_1$ a projective line
\[
Y_K  \simeq \bP^1_K,
\]
but as an $R$-scheme $Y/R$ is only locally of finite type. 

An alternative, but less symmetric, way to construct $Y$ starts with $Y_0 = \bP^1_R$ and blows up all $k$-rational points of the special fibre to obtain $Y_1 \to Y_0$. This $Y_1$ has extra components in the special fibre above every $k$-rational point. Inductively we define $Y_{n+1}$ as the blow up of $Y_n$ in all $k$-rational points that do not lie on the strict transform of a component of $Y_{n-1}$.   In this way, the \textit{old} part of $Y_n$ becomes stable and $Y = \varprojlim_n Y_n$ is locally of finite type.

\subsubsection{The group action} The $R$-scheme $Y$ carries an action
\[
\PGL_2(K) \to \Aut_R(Y)
\]
defined as follows. We let $S = (S_{ij}) \in \GL_2(K)$ act in the standard way on $\bP^1_K = Y_K$ via the linear action on homogeneous coordinates $x_0$, $x_1$ with $x =\genfrac{(}{)}{0pt}{}{x_0}{x_1}$ by 
\[
S^{\ast}(x) := S^{-1} x.
\]
This action of the generic fibre $Y_K$ extends to the model $Y$ by permuting the open patches as
\[
\xymatrix@M+1ex@R-2ex{
Y_L \ar[r]^S & Y_{LS} \\
Y_K \ar@{}[u]|{\displaystyle \cup} \ar[r]^S & Y_{K} \ar@{}[u]|{\displaystyle \cup}
}
\]
since
\[
S^\ast \genfrac{(}{)}{0pt}{}{\ell^{LS}_0}{\ell^{LS}_1} =  S^\ast(LSx) = LS(S^\ast x) = LS(S^{-1}x) = Lx =  
\genfrac{(}{)}{0pt}{}{\ell^L_0}{\ell^L_1}.
\]
This action by $\GL_2(K)$ descends to an action by $\PGL_2(K)$ on $Y/R$.

\subsubsection{The dual graph} 
The \textit{dual graph} of the special fibre $Y_s$ of $Y/R$ has a vertex $v_C$ for each irreducible component $C$ of $Y_s$ and an unoriented edge for each double point $P$ that joins the two vertices $v_{C}$, $v_{C'}$  representing the components $C$, $C'$ with $\{P\} = C \cap C'$.

\begin{lem}
The dual graph of $Y_s$ is isomorphic to the Bruhat Tits tree $T$ of $\PGL_2(K)$ as a graph with action by $\PGL_2(K)$.
\end{lem}
\begin{proof}
There is a natural $\PGL_2(K)$-equivariant bijection between the following three sets:
\begin{enumeral}
\item Irreducible components of the special fibre $Y_s$.
\item $R$-lattices $M$ in $\rH^0(\bP^1_K,\dO(1)) = K \cdot x_0 \oplus K \cdot x_1$ up to homothety, i.e., free $R$-submodules of rank $2$ up to scaling by $K^\ast$.
\item Vertices of $T$.
\end{enumeral}
Indeed, a component of $Y_s$ is given via $L \in \GL_2(K)$ with $Lx = \genfrac{(}{)}{0pt}{}{\ell_0}{\ell_1}$ as the strict transform of the special fibre of 
\[
Y \xrightarrow{}  P_L = \Proj R[\ell_0,\ell_1] \simeq  \bP^1_R
\]
where the map  is the identity on function fields $K(\frac{x_0}{x_1}) = K(\frac{\ell_0}{\ell_1})$.  This depends on $L$ only through its class in 
\[
\GL_2(K)/K^\ast \cdot \GL_2(R) = \PGL_2(K)/\PGL_2(R).
\]
The corresponding homothety class of lattices is 
\[
M_L = R \cdot \ell_0 \oplus R \cdot \ell_1 = \rH^0(P_L,\dO(1)) \subset \rH^0(\bP^1_K,\dO(1)).
\]

Two components $C_1,C_2 \subset Y_s$ intersect in a point $P$ if and only if $Y$ is given in a neighbourhood of $P$ by $Y_L$ with the two components of the special fibre $Y_{L,s}$ having $C_1$ and $C_2$ as closure. This holds if and only if the corresponding lattices are of the form 
\[
\pi M_1  \subsetneq  M_{2}= R \cdot \pi \ell_1 \oplus R \cdot \ell_0 \subsetneq M_1 = R \cdot \ell_0 \oplus R \cdot \ell_1
\]
for $\ell_0,\ell_1$ as above determined by $L$, if and only if the corresponding vertices in the Bruhat Tits tree are joined by an edge.

Equivariance under $\PGL_2(K)$ is obvious by the naturality of the construction.
\end{proof}

\subsection{The wonderful scheme squared}
We are actually interested in $Y \times_R Y$.

\subsubsection{Local structure}
The $R$-scheme $Y \times_R Y$ is locally isomorphic to 
\[
R[u,v,w,z]/(uv = \pi = wz) = R[u,v]/(uv -\pi) \otimes_R R[w,z]/(wz-\pi).
\]
Since $R[u,v]/(uv-\pi)$ is regular and even smooth over $R$ outside of $(0,0)$, the only singular points of $Y \times_R Y$ correspond in the local model to $(u,v,w,z) = (0,0,0,0)$. There $u,v,w+z$ forms a regular sequence of length $3 = \dim(Y \times_R Y)$, so that $Y  \times_R Y$ is normal, Cohen-Macauley, and even regular outside the $k$-rational points of its special fibre.

If we blow-up the singular locus, then every singular point is replaced by a copy of $\bP^1_k \times \bP^1_k$ and the special fibre of the blow-up  has normal crossing.

\subsubsection{The dual square complex} 

We define a \textit{dual square complex} $\Sigma$ describing the combinatorics of the components of its special fibre. 
\begin{enumeral}
\item 
Vertices of $\Sigma$ are irreducible components of the special fibre 
\[
(Y \times_R Y)_s = Y_s \times_k Y_s.
\]
These are isomorphic to  $\bP^1_k \times \bP^1_k$.
\item
Unoriented edges of $\Sigma$ are irreducible curves in the intersection of irreducible components of $Y_s$ and the edge joins the corresponding vertices. Such intersections occur precisely  along the grid of vertical and horizontal lines
\[
\bP^1_k \times \bP^1(k) \cup \bP^1(k) \times \bP^1_k \  \subset  \ \bP^1_k \times \bP^1_k
\]
with the intersection being transversal outside of $\bP^1(k) \times \bP^1(k)$. 
\item Squares  of $\Sigma$ are defined by singular points of $Y \times_R Y$. These are the $k$-rational points of the special fibre, i.e., on a component the points $P \in \bP^1(k) \times \bP^1(k)$. There are four components of $Y_s \times_k Y_s$ passing through each $P$ and there $Y \times_R Y$ is locally isomorphic to 
\[
\Spec\big(R[u,v,w,z]/(uv = \pi = wz)\big)
\]
with $P$ mapping to the non-regular point $(0,0,0,0)$. 
\end{enumeral}

\begin{lem}
The product $Y \times_R Y$ carries a natural action by $\PGL_2(K) \times \PGL_2(K)$ such that the dual square complex of its special fibre agrees with the square complex $T \times T$ together with its Bruhat Tits action by $\PGL_2(K) \times \PGL_2(K)$.
\end{lem}
\begin{proof}
The dual square complex $\Sigma$ is the product of the dual graph of $Y_s$ with itself. 
\end{proof}

\subsection{The formal geometric quotient} Consider the formal scheme 
\[
\cY \times_R \cY \to \Spf(R)
\]
obtained by completing $Y \times_R Y$ along its special fibre. The generic fibre, in the sense of rigid geometry over $K$, is $\bP^1_K \times \bP^1_K$, hence a smooth projective rigid variety.

\smallskip

Now let $\Gamma \subset \PGL_2(K) \times \PGL_2(K)$ be a discrete torsion free group that acts cocompactly on $T \times T$ via the Bruhat Tits action. The induced action on $\cY \times_R \cY$ is free and discontinuously on the underlying Zariski topology, namely the Zariski topology of $Y_s \times_k Y_s$. Indeed, the underlying dual square complex is $\Sigma = T \times T$ and $\Gamma$ acts already with trivial stabilizers on vertices, edges or squares of $\Sigma$. It is thus trivially possible to pass to the quotient 
\[
\cX_\Gamma = \Gamma \backslash \cY \times_R \cY
\]
as a formal scheme over $\Spf(R)$. The dual square complex of $\cX$ is the finite quotient complex
\[
\Sigma(\cX_\Gamma) = \Gamma \backslash \Sigma.
\]
Thus $\cX$ is of finite type and proper as a formal scheme over $\Spf(R)$.
Let $N$ be the number of vertices (in some sense the volume) of $\Sigma(\cX_\Gamma)$. 

\begin{lem} \label{lem:countingcombinatorics}
The number of unoriented edges and squares in the square complex $\Sigma(\cX_\Gamma)$ is
\begin{enumera}
\item $\#\ov{E}(\Sigma(\cX_\Gamma)) = N  (q+1)$,
\item $\#S(\Sigma(\cX_\Gamma)) = N  (q+1)^2/4$, 
\end{enumera}
and the geometric realization $\Sigma(\cX_\Gamma)$ has Euler--characteristic 
\[
\chi(\Sigma(\cX_\Gamma)) = N (q-1)^2/4.
\]
\end{lem}
\begin{proof}
The formulae for the number of edges and squares follow by counting in two ways the number of "vertices in the boundary of edges":
\[
2 \#\ov{E}(\Sigma(\cX_\Gamma))  =  2(q+1) N,
\]
and of "edges in the boundary of squares":
\[
4 \#S(\Sigma(\cX_\Gamma)) = (q+1)^2 N.
\]
Then the  Euler characteristic is 
\[
\chi(\Sigma(\cX_\Gamma)) = N - N(q+1) +  N(q+1)^2/4 =  N(q-1)^2/4.
\]
\end{proof}

\subsection{Algebraization}
The sheaf of relative log-differentials (with respect to the log structure defined by the special fibre)
\[
\Omega^{1,\log}_{Y/R}
\]
is a line bundle on $Y$ that locally on $Y_L$ is generated in $\Omega^1_{K(x_0/x_1)/K}$ by 
\[
d \log\left(\frac{\ell_0}{\ell1}\right), \quad d \log\left(\frac{\pi\ell_1}{\ell_0}\right)
\]
and the action by $\PGL_2(K)$ on $Y$ extends to an action on $\Omega^{1,\log}_{Y/R}$.
The restriction of $\Omega^{1,\log}_{Y/R}$ to a component $\bP^1_k$ of $Y_s$ takes the form
\[
\Omega^1_{\bP^1_k/k}(\bP^1(k)) \simeq \dO(q - 1).
\]

Let us consider the exterior tensor product  on $Y \times_R Y$
\[
\Omega^{2,\log}_{Y \times_R Y/R} = \Omega^{1,\log}_{Y/R} \boxtimes \Omega^{1,\log}_{Y/R},
\]
that, when restricted to an irreducible component $\bP^1_k \times_k \bP^1_k  \simeq C \subset Y_s \times_k Y_s$, becomes 
\begin{equation} \label{eq:omega2logoncomponent}
\Omega^{2,\log}_{Y \times_R Y/R}|_C \simeq \dO(q-1) \boxtimes \dO(q-1).
\end{equation}

\begin{prop} \label{prop:algebraization}
\begin{enumera}
\item
The formal scheme $\cX_\Gamma/R$ is the formal completion along the special fibre of a projective $R$-scheme $X_\Gamma/R$.
\item The generic fibre $X_{\Gamma,K} = X_\Gamma \otimes_R K$ is smooth projective with ample canonical bundle.
\item In particular, $X_{\Gamma,K}$ is a minimal surface of general type without smooth rational curves $E \subset X_{\Gamma,K}$ with self intersection $(E^2) = -1$ or $-2$.
\end{enumera}
\end{prop}
\begin{proof}
(1)
Since $\cX_\Gamma/\Spf(R)$ is proper, it suffices by Grothendieck's formal GAGA to find a relatively ample line bundle on $\cX_\Gamma$. 
The pull back to the formal completion $\Omega^{2,\log}_{Y \times_R Y/R}|_{\cY \times_R \cY}$
descends to a line bundle 
\[
\Omega^{2,\log}_{\cX_\Gamma/R}
\]
on $\cX_\Gamma$ which is ample, because it restricts to an ample line bundle $\dO(q-1) \boxtimes \dO(q-1)$ 
on the normalization of every irreducible
 component of $\cX_\Gamma$ by \eqref{eq:omega2logoncomponent}. It follows that $\cX_\Gamma$ can be algebraized 
 to a projective scheme $X_\Gamma/R$. 

(2) The generic fibre $(Y \times_R Y)_K$ is smooth over $K$. Therefore also the generic fibre $X_{\Gamma,K}/K$ is smooth, because  
$\cY \times_R \cY \to \cX_\Gamma$ is a Zarsiki locally trivial covering map and smoothness of the generic fibre can be tested locally on the formal completion along the special fibre.
Moreover, the canonical sheaf $\omega_{X_{\Gamma,K}/K} = \Omega^{2,\log}_{X_\Gamma/R} |_{X_{\Gamma,K}}$ 
is ample as the restriction of the relatively ample line bundle $\Omega^{2,\log}_{\cX_\Gamma/R}$.

Assertion (3) follows at once from (2) since smooth rational curves $E$ with $(E^2) = -1, -2$ are precisely those which are contracted in the map to the canonical model.
\end{proof}

\subsection{Chern numbers}
We compute the Chern numbers of $X_\Gamma$ by degeneration to the special fibre 
$X_{\Gamma,s} = X_\Gamma \otimes_R k$, a proper $k$-variety. Let
\begin{enumera}
\item $\pi_E: \tilde{E} \to E$ be the normalization of an irreducible component $E \subset X_{\Gamma,s}$, hence a vertex $E \in V(\Sigma(\cX_\Gamma))$, and $\tilde{E} \simeq \bP^1_k \times_k \bP^1_k$,
\item $\pi_C: \tilde{C} \to C$ be the normalization of an irreducible curve in $X_{\Gamma,s}$ that represents an unoriented edge $C \in \ov{E}(\Sigma(\cX_\Gamma))$, and
$\tilde{C} \simeq \bP^1_k$, 
\item $i_P : P \inj X_{\Gamma,s}$ be the $k$-rational point corresponding to a square $P \in S(\Sigma(\cX_\Gamma))$.
\end{enumera}
If the vertices $E_1,E_2$ are joined by an edge $C$, then there are natural immersions 
\begin{equation} \label{eq:res1}
\tilde{E}_1 \hookleftarrow  \tilde{C} \inj \tilde{E}_2,
\end{equation}
one as vertical and the other as horizontal line.
If  $C$ is an edge in the boundary of a square $P$, then associated to this is an inclusion 
\begin{equation} \label{eq:res2}
P \inj \tilde{C}.
\end{equation}
We now choose an orientation on $\Sigma(\cX_\Gamma)$. This induces signs for every incidence relation, a vertex in an edge, an edge in a square, such that the complex
\begin{equation} \label{eq:choiceofsigns}
\bigoplus_{E \in V(\Sigma(\cX_\Gamma))} \bZ  \xrightarrow{\pm \res} \bigoplus_{C \in \bar{E}(\Sigma(\cX_\Gamma))} \bZ \xrightarrow{\pm\res} \bigoplus_{P \in S(\Sigma(\cX_\Gamma))}\bZ
\end{equation}
becomes the cellular complex of the CW-complex $\Sigma(\cX_\Gamma)$. 

\begin{prop} \label{prop:resolution}
With signs as in \eqref{eq:choiceofsigns} and restriction induced by the inclusions \eqref{eq:res1} and \eqref{eq:res2} the following sequence  is exact:
\[
0 \to \dO_{X_{\Gamma,s}} \to \bigoplus_{E \in V(\Sigma(\cX_\Gamma))} \pi_{E,\ast} \dO_{\tilde{E}}  \xrightarrow{\pm\res} \bigoplus_{C \in \bar{E}(\Sigma(\cX_\Gamma))} \pi_{C,\ast} \dO_{\tilde{C}} \xrightarrow{\pm\res} \bigoplus_{P \in S(\Sigma(\cX_\Gamma))} i_{P,\ast} \dO_P \to 0.
\]
\end{prop}
\begin{proof}
The sequence lives on $X_{\Gamma,s} = \cX_{\Gamma,0}$, the special fibre of the formal $R$-scheme $\cX_\Gamma$. Being a Zariski local issue exactness can be checked after pull back to $(\cY \times_R \cY)_0 = Y_s \times_k Y_s$. There exactness is clear  in the locus where the components  meet as normal crossing divisors on $Y$. It remains to check exactness in a singular point $P$.

In $P$ we exploit the product structure: locally we need to consider the spectrum of 
\[
k[u,v,x,y]/(uv,xy) = k[u,v]/(uv) \otimes_k k[x,y]/(xy).
\]
Now the sequence built from restriction maps 
\[
0 \to k[u,v]/(uv) \to k[u] \oplus k[v] \xrightarrow{\pm} k \to 0
\]
is exact, so that 
\[
k[u,v,x,y]/(uv,xy) \to \Tot \left(\left[k[u] \oplus k[v] \xrightarrow{\pm} k\right] \otimes_k \left[k[x] \oplus k[y] \xrightarrow{\pm} k\right]\right)
\]
is a quasi-isomorphism. The corresponding resolution of $k[u,v,x,y]/(uv,xy)$ is precisely our complex locally in a neighbourhood of $P$ resp.\ $(0,0,0,0)$.
\end{proof}

Since $X_\Gamma/R$ is flat and projective, the Euler characteristic of its fibres is constant. 
We conclude by Proposition~\ref{prop:resolution} and Lemma~\ref{lem:countingcombinatorics} that
\begin{eqnarray} \label{eq:chiX_K}
\chi(X_{\Gamma,K},\dO_{X_{\Gamma,K}}) & = & \chi(X_{\Gamma,s},\dO_{X_{\Gamma,s}}) \\
& = & \#V(\Sigma(\cX_\Gamma)) \chi(\bP^1 \times \bP^1,\dO) - \#\ov{E}(\Sigma(\cX_\Gamma)) \chi(\bP^1,\dO) + \#S(\Sigma(\cX_\Gamma)) \chi(P,\dO) \notag \\
& = & \chi(\Sigma(\cX_\Gamma)) = N  (q-1)^2/4. \notag
\end{eqnarray}
Similarly, we  compute the square of the canonical class
\begin{eqnarray} \label{eq:c1squared} 
\rc_1(X_{\Gamma,K})^2  & = & \big(\rc_1(\Omega^{2,\log}_{\cX_\Gamma/R})|_{X_{\Gamma,s}}\big)^2 \\
& = & \sum_{E \in V(\Sigma(\cX_\Gamma))} \big(\rc_1(\Omega^{2,\log}_{\cX_\Gamma/R})|_{\tilde{E}}\big)^2 = N \cdot \big((\dO(q-1) \boxtimes \dO(q-1))^2\big) \notag \\
& = & 2N(q-1)^2. \notag
\end{eqnarray}
Noether's formula  then yields
\begin{eqnarray} \label{eq:c2} 
\rc_2(X_{\Gamma,K})  & = & 12\chi(X_{\Gamma,K},\dO_{X_{\Gamma,K}})  - \rc_1(X_{\Gamma,K})^2   = N (q-1)^2. 
\end{eqnarray}

\subsection{The Albanese variety via Kummer \'etale cohomology}
We endow $X_\Gamma/R$ with the fs-log structure in the sense of Fontaine and Illusie determined by its special fibre. The resulting log-scheme is log-smooth and projective over $\Spec(R)$ endowed with its canonical log-structure. Let $\bar K$ be an algebraic closure of $K$ and $X_{\Gamma,\bar K}$ the corresponding  geometric generic fibre. Let $\tilde{s} \to \Spec(R)$ be a log-geometric point over the closed point and $X_{\Gamma,\tilde{s}}$ the corresponding log geometric special fibre.

\smallskip

Let $\Lambda$ be a finite ring of order prime to $p$, for example $\Lambda = \bF_\ell$ for a prime $\ell \not= p$. Since $X_\Gamma/R$ is log-smooth, the Kummer-\'etale cospecialisation map is an isomorphism
\[
\rH^1_{\et}(X_{\Gamma,\bar K},\Lambda)  \simeq \rH^1_{\ket}(X_{\Gamma,\tilde{s}},\Lambda).
\]
We compute $\rH^1_\ket(X_{\Gamma,\tilde{s}},\Lambda)$ by means of a resolution in sheaves on the Kummer \'etale site 
$(X_{\Gamma,s})_{\ket}$ 
\[
0 \to \Lambda \to \bigoplus_{E \in V(\Sigma(\cX_\Gamma))} \pi_{E,\ast} \Lambda  \xrightarrow{\pm \res} \bigoplus_{C \in \bar{E}(\Sigma(\cX_\Gamma))} \pi_{C,\ast} \Lambda  \xrightarrow{\pm \res} \bigoplus_{P \in S(\Sigma(\cX_\Gamma))} i_{P,\ast} \Lambda \to 0
\]
with signs coming from \eqref{eq:choiceofsigns} and using the hypercohomology spectral sequence as:
\begin{equation} \label{eq:computeH1}
0 \to \rH^1_{\rm Sing}(\Sigma(\cX_\Gamma),\Lambda) \to \rH^1_\ket(X_{\Gamma,\tilde{s}},\Lambda) \to \bigoplus_{E \in V(\Sigma(\cX_\Gamma))}
 \rH^1_\ket(\tilde{E}_{\tilde{s}},\Lambda) \xrightarrow{\pm \res} \bigoplus_{C \in E(\Sigma(\cX_\Gamma))}  \rH^1_\ket(\tilde{C}_{\tilde{s}},\Lambda).
\end{equation}
Starting from now we assume that $K = \bF_q((t))$ is in odd equi-characteristic $p$ and 
\[
\Gamma = \Gamma_\tau
\]
is one of the irreducible arithmetic lattices described by 
Theorem ~\ref{thm:presentationgamma}. Hence 
\[
\Sigma(\cX_\Gamma) = S_{\Gamma_\tau}
\]
has only one vertex and a VH-structure coming from a VH-structure $A,B$ in the group $\Gamma$. Implicitly we have chosen an identification $A = \bP^1(k) = B$, and then  
\begin{eqnarray*}
\rH^1_\ket(\tilde{E}_{\tilde{s}},\Lambda)  & = & \big(\bigoplus_{a \in A} \Lambda(-1)\big)^{0}  \oplus \big(\bigoplus_{b \in B} \Lambda(-1)\big)^{0}, \\
\rH^1_\ket(\tilde{C}_{\tilde{s}},\Lambda)  & = & \big(\bigoplus_{\xi \in \bP^1(k)} \Lambda(-1)\big)^{0}. 
\end{eqnarray*}
Here $\Lambda(-1)$ is the inverse Tate-twist and $(-)^{0}$ denotes the kernel of the summation map.

\begin{lem} \label{lem:computeH1}
If $q+1 \in \Lambda^\ast$, then $\rH^1_{\et}(X_{\Gamma,\bar K},\Lambda) \simeq \Hom(\Gamma^\ab, \Lambda)$.
\end{lem}
\begin{proof}
The maps $\pm \res$ in \eqref{eq:computeH1} impose invariance under local permutation groups $P_A$ and $P_B$, 
see Section~\S\ref{sec:localstructure}, so that we find an exact sequence
\[
0 \to\rH^1_{\rm Sing}(\Sigma(\cX_\Gamma),\Lambda)  \to \rH^1_\ket(X_{\Gamma,\tilde{s}},\Lambda) \to 
\rH^0(P_A,\big(\bigoplus_{a \in A} \Lambda(-1)\big)^{0})  \oplus \rH^0(P_B,\big(\bigoplus_{b \in B} \Lambda(-1)\big)^{0}) \to 0.
\]
Since the local permutation groups are $2$-transitive by Proposition~\ref{prop:determinelocalstructure}, the right hand side vanishes if $q+1 \in \Lambda^\ast$, and finally under this additional assumption 
\[
\rH^1_{\et}(X_{\Gamma,\bar K},\Lambda) \simeq \rH^1_\ket(X_{\Gamma,\tilde{s}},\Lambda) \simeq \rH^1_{\rm Sing}(\Sigma(\cX_\Gamma),\Lambda) \simeq \Hom(\Gamma^\ab, \Lambda)
\]
as claimed.
\end{proof}

\begin{prop} \label{prop:trivialalb}
The Albanese variety of $X_{\Gamma_\tau,K}$ is trivial.
\end{prop}
\begin{proof}
By Lemma~\ref{lem:computeH1} the prime to $p(q+1)$ maximal abelian quotients of $\pi^\et_1(X_{\Gamma,\bar K})$ and $\Gamma$ agree. Since the latter is finite by Proposition~\ref{prop:justinfinite}  the Albanese variety of $X_{\Gamma,K}$ must be trivial.
\end{proof}

\subsection{On the \'etale fundamentalgroup}
In order to prove Theorem~\ref{thm:thmC} of the introduction we must analyse the \'etale fundemental group of $X_{\Gamma_\tau,K}$. 

\begin{thm} \label{thm:ncfakequadric} 
Let $\Gamma_\tau$ be one of the lattices described in Theorem~\ref{thm:presentationgamma} above. The surface 
\[
X_\tau = X_{\Gamma_\tau}
\]
is a minimal smooth projective surface of general type over $\bF_q((t))$ with 
\begin{enumeral}
\item ample canonical bundle, 
\item Chern ratio $\rc_1^2/\rc_2 = 2$, 
\item trivial Albanese variety, 
\item \label{item:nonredpic} non-reduced Picard scheme of dimension $0$, and 
\item \label{item:pi1quot} geometric \'etale fundamental group with an infinite continuous quotient
\[
\ov{\pi}_1^{\et}(X_\tau) \surj \widehat{\Gamma}_\tau.
\]
\end{enumeral}
For $q = 3$, the surface $X_\tau$ is a non-classical fake quadric over $\bF_3((t))$.
\end{thm}
\begin{proof}
The existence of $X_\tau$ as a minimal smooth projective surface of general type was proven in Proposition~\ref{prop:algebraization}, the Chern numbers are computed in \eqref{eq:c1squared}  and \eqref{eq:c2}, and Proposition~\ref{prop:trivialalb} deals with the claim on the Albanese variety.  It remains to show \ref{item:pi1quot} and \ref{item:nonredpic}.

Consider $\cX_\tau = \cX_{\Gamma_\tau}$ with geometric special fibre $X_{\tau,\bar{s}}$. Since $\cX_\tau$ has genrically a reducd semistable geometrically connected special fibre, the geometric specialisation map 
\[
{\rm sp}: \ov{\pi}_1^{\et}(X_\tau) = \pi_1^{\et}(X_{\tau,\bar K}) \surj \pi_1^{\et}(X_{\tau,\bar{s}})
\]
is surjective. The formal Zariski-trivial cover $\cY \times_R \cY \to \cX_\tau$ induces a continuous quotient map 
\[
\pi_1^{\et}(X_{\tau,\bar{s}}) \surj \widehat{\Gamma}_\tau
\]
to the profinite completion $\widehat{\Gamma}_\tau$ of $\Gamma_\tau$. The composition of these two quotient maps 
yields the map in \ref{item:pi1quot}, and $\widehat{\Gamma}_\tau$ is infinite because $\Gamma_\tau$ is 
residually pro-$p$ due to Proposition~\ref{prop:resprop}.

For \ref{item:nonredpic} we note that $\dim(\Pic^0_{X_\tau})$ equals the dimension of the Albanese variety, but that $\Pic^0$ is nontrivial due to non-trivial $p$-torsion quotients, i.e., by Proposition~\ref{prop:uniformabelianquot},
\[
\pi_1^{\ab}(X_{\tau,\bar K}) \surj \Gamma^\ab \surj \bF_q[Z]
\]
which leads to non-trivial multiplicative torsion in $\Pic_{X_\tau}$ hence in $\Pic_{X_\tau}^0$. In particular
\[
\dim \rH^1(X_{\tau,K},\dO_{X_\tau})  = \dim (\rT_0 \Pic^0_{X_\tau}) \geq 2 \log_p(q) \geq 2
\]
is nontrivial.

If $q=  3$, then $\chi = 1$ by Lemma~\ref{lem:countingcombinatorics} and thus $\rc_1^2 = 8$ and $\rc_2 = 4$, so that indeed $X_\tau$ is a non-classical fake quadric in this case. 
\end{proof}




\begin{thebibliography}{[BM00a]}

\bibitem[BB95]{ballmannbrin}
Ballmann, W., Brin, M.,
Orbihedra of nonpositive curvature,
\textit{Inst.\ Hautes Etudes Sci.\ Publ.\ Math.} \textbf{82} (1995), 169--209.

\bibitem[BB12]{boecklebutenuth}
B\"ockle, G., Butenuth, R.,
On computing quaternion quotient graphs for function fields,
\textit{J.\ Th\'eor.\ Nombres Bordeaux} \textbf{24} (2012), no.\ 1, 73--99. 

\bibitem[BCG05]{bauercatanesegrunewald}
Bauer, I.C. and Catanese, F. and Grunewald, F.,
The classification of surfaces with $p_{g}=q=0$ isogenous to a product of curves,
\textit{Pure {A}ppl. {M}ath. Q.} \textbf{4} special Issue: In honor of Fedor Bogomolov Part 1, (2008), no.\ 2, 547--586.

\bibitem[Be96]{beauville:surfaces}
Beauville, A., 
\textit{Complex algebraic surfaces},
translated from the 1978 French original by R. Barlow, with assistance from N. I. Shepherd-Barron and M. Reid,
London Mathematical Society Student Texts \textbf{34} (2nd ed.), Cambridge University Press, 1996, x+132 pp.

\bibitem[Bh98]{behr}
Behr, H.,
Arithmetic groups over function fields I. A complete characterization of finitely generated and finitely presented arithmetic subgroups of reductive algebraic groups,
\textit{J.\ Reine Angew.\ Math.} \textbf{495} (1998), 79--118. 

\bibitem[Bl03]{blumenthal}
Blumenthal, O., 
\"Uber Modulfunktionen von mehreren Ver\"anderlichen, 
\textit{Math.\ Ann.} \textbf{56} (1903), 509--548.

\bibitem[BLR90]{blr:neronmodels}
Bosch, S., L\"utkebohmert, W., Raynaud, M.,
\textit{N\'eron models}, 
Ergebnisse der Mathematik und ihrer Grenzgebiete (3), volume \textbf{21}, Springer-Verlag, Berlin, 1990, x+325 pp.

\bibitem[BM97]{burger-mozes:simple}
Burger, M., Mozes, Sh.,
Finitely presented simple groups and products of trees,
\textit{C.\ R.\ Acad.\ Sci.\ Paris S\'er.\ I Math.} \textbf{324} (1997), no.\ 7, 747--752. 

\bibitem[BM00a]{burger-mozes:localtoglobal}
Burger, M., Mozes, Sh.,
Groups acting on trees: from local to global structure,
\textit{Inst.\ Hautes \'Etudes Sci.\ Publ.\ Math.} \textbf{92} (2000), 113--150.

\bibitem[BM00b]{burger-mozes:lattices}
Burger, M., Mozes, Sh.,
Lattices in product of trees,
\textit{Inst.\ Hautes \'Etudes Sci.\ Publ.\ Math.} \textbf{92} (2000), 151--194.

\bibitem[Di22]{dickson:quaternions}
Dickson, L.\thinspace{}E., 
Arithmetic of quaternions, 
\textit{Proc.\ London Math.\ Soc.} (2) \textbf{20} (1922), 225--232.

\bibitem[Di58]{dickson:pgl2}
Dickson, L.\thinspace{}E.,
\textit{Linear groups: With an exposition of the Galois field theory},
with an introduction by W. Magnus, Dover Publications, Inc., New York 1958, xvi+312 pp. 

\bibitem[Dz12]{dzambic:fakequadrics}
D\v{z}ambi\'c, A., 
Fake quadrics from irreducible lattices acting on the product of upper half planes, preprint 2012.

\bibitem[Fa12]{faber:groupsinpgl2}
Faber, X., 
Finite $p$-Irregular Subgroups of $\PGL(2,k)$, preprint, \href{http://arxiv.org/abs/1112.1999v2}{arXiv:[math.NT]1112.1999v2}.

\bibitem[Gi07]{giudici}
Giudici, M.,
Maximal subgroups of almost simple groups with socle $\PSL(2,q)$,
preprint, \href{http://arxiv.org/abs/math/0703685v1}{arXiv: [math.GR]0703685v1}.

\bibitem[GN95]{gekelernonnengardt}
Gekeler, E.-U.,  Nonnengardt, U.,
Fundamental domains of some arithmetic groups over function fields, 
\textit{Internat.\ J.\ Math.} \textbf{6} (1995), 689--708.

\bibitem[He54]{herrmann}
Herrmann, O., 
Eine metrische Charakterisierung eines Fundamentalbereiches der Hilbertschen Modulgruppen,
\textit{Math.\ Z.} \textbf{60} (1954), 148--155.

\bibitem[Hi87]{hirzebruch:oevre1}
Hirzebruch, F.,
\textit{Gesammelte Abhandlungen, Band I, 1951--1962},
Springer-Verlag, Berlin, 1987, viii + 814 pp.

\bibitem[KW80]{kirchheimerwolfahrt}
Kirchheimer, F., Wolfart, J.,
Explizite Pr\"asentation gewisser Hilbertscher Modulgruppen durch Erzeugende und Relationen,
\textit{J.\ Reine Angew.\ Math.} \textbf{315} (1980), 139--173.

\bibitem[Ma40]{maass:hilbertmodulargroup}
Maass, H., 
\"Uber Gruppen von hyperabelschen Transformationen, 
\textit{Sitzungsber.\ Heidelberg Akad.\ Wiss.\ Math.-nat.\ Klasse \textbf{2}}, 1940.

\bibitem[Ma91]{margulis:book}
Margulis, G.~A.,
\textit{Discrete subgroups of semisimple Lie groups},
Ergebnisse der Mathematik und ihrer Grenzgebiete (3), volume \textbf{17}, Springer-Verlag, Berlin, 1991,  x+388 pp.

\bibitem[Mo95]{mozes:cartan}
Mozes, Sh.,
Actions of Cartan subgroups,
\textit{Israel J.\ Math.} \textbf{90} (1995), no.\ 1--3, 253--294. 

\bibitem[Mu79]{mumford:fakep2}
Mumford, D.,
An algebraic surface with $K$ ample, $(K^2)=9$, $p_g = q = 0$,
\textit{Amer.\ J.\ Math.} \textbf{101} (1979), no.\ 1, 233--244.

\bibitem[Pa11]{papikian:GLn}
Papikian, M.,
On finite arithmetic simplicial complexes,
\textit{Proc.\ Amer.\ Math.\ Soc.} \textbf{139} (2011), no.\ 1, 111--124. 

\bibitem[Pa12]{papikian:Delliptic}
Papikian, M.,
Local diophantine properties of modular curves of $\cD$-elliptic sheaves
\textit{J.\ reine angew.\ Math.} \textbf{664} (2012), 115--140.

\bibitem[Ra04]{rattaggi:thesis}
Rattaggi, D.\thinspace{}A., Computations in Goups acting on a product of trees: normal subgroup structures and quaternion lattices, thesis ETH Z\"urich, 2004.

\bibitem[Se80]{serre:trees} 
Serre, J.-P., 
{\it Trees}, translated from the French by John Stillwell, 
Springer, 1980, ix+142pp.

\bibitem[Sh78]{shavel:fakequadric}
Shavel, I.\thinspace{}H.,
A class of algebraic surfaces of general type constructed from quaternion algebras,
\textit{Pacific J.\ Math.} \textbf{76} (1978), no.\ 1, 221--245. 

\bibitem[Sw71]{swan:bianchigrouppresentation}
Swan, R.\thinspace{}G.,
Generators and relations for certain special linear groups,
\textit{Advances in Math.} \textbf{6} (1971), 1--77.

\bibitem[Vo09]{voight:funddomainfuchsian}
Voight, J.,
Computing fundamental domains for Fuchsian groups,
\textit{J.\ Th\'eor.\ Nombres Bordeaux} \textbf{21} (2009), no.\ 2, 469--491. 

\bibitem[Wi96]{wise:thesis}
Wise, D.,
{Non-positively curved squared complexes, aperiodic tilings, and non-residually finite groups},  Ph.D.~thesis, Princeton University,  1996, 71 pp.

\end{thebibliography}
\end{document}